\documentclass[reqno]{amsart}
\usepackage{amssymb, latexsym}
\usepackage{hyperref}

\allowdisplaybreaks
\newcommand{\N}{{\mathbb{N}}}

\newcommand{\C}{{\mathbb{C}}}
\newcommand{\R}{{\mathbb{R}}}

\let\Re=\undefined\DeclareMathOperator*{\Re}{Re}

\DeclareMathOperator*{\supp}{supp}  
\DeclareMathOperator{\osc}{osc}

\newcommand{\eps}{{\varepsilon}}

\def\rootsixfifths{\sqrt{\smash{\frac65}\vrule depth 0.5ex width 0mm height 2ex}}

\newcommand{\gnj}{g_n^j}
\newcommand{\tnj}{t_n^j}

\newcommand{\wnJ}{w_n^J}

\newcommand{\vnj}{v_n^j}

\theoremstyle{plain}
\newtheorem{theorem}{Theorem}
\newtheorem{proposition}[theorem]{Proposition}
\newtheorem{lemma}[theorem]{Lemma}

\newtheorem{conjecture}[theorem]{Conjecture}
\theoremstyle{definition}
\newtheorem{definition}[theorem]{Definition}
\newtheorem{remark}[theorem]{Remark}

\newcounter{smalllist}

\newenvironment{CI}{\begin{list}{{\ $\bullet$\ }}{%
\setlength{\topsep}{0mm}\setlength{\parsep}{0mm}\setlength{\itemsep}{0mm}%
\setlength{\labelwidth}{0mm}\setlength{\itemindent}{0mm}\setlength{\leftmargin}{0mm}%
\setlength{\labelsep}{0mm}}}{\end{list}}

\numberwithin{equation}{section} \numberwithin{theorem}{section} \voffset=-0.5in  \hoffset=-0.2in
\begin{document}

\title[On the mass-critical gKdV equation]{On the mass-critical generalized KdV equation}
\author{Rowan Killip}
\address{University of California, Los Angeles}
\author{Soonsik Kwon}
\address{Princeton University, Princeton, NJ}
\author{Shuanglin Shao}
\address{Institute for Advanced Study, Princeton, NJ}
\author{Monica Visan}
\address{University of California, Los Angeles}

\begin{abstract}
We consider the mass-critical generalized Korteweg--de Vries equation
$$(\partial_t + \partial_{xxx})u=\pm \partial_x(u^5)$$
for real-valued functions $u(t,x)$. We prove that if the global well-posedness and scattering conjecture for this equation failed,
then, conditional on a positive answer to the global well-posedness and scattering conjecture for the mass-critical nonlinear
Schr\"odinger equation $(-i\partial_t + \partial_{xx})u=\pm (|u|^4u)$, there exists a minimal-mass blowup solution to the
mass-critical generalized KdV equation which is almost periodic modulo the symmetries of the equation.  Moreover, we can
guarantee that this minimal-mass blowup solution is either a self-similar solution, a soliton-like solution, or a double high-to-low
frequency cascade solution.
\end{abstract}

\maketitle

\section{Introduction}
We consider the initial-value problem for the mass-critical generalized Korteweg--de Vries (gKdV) equation
\begin{gather}
(\partial_t+\partial_{xxx})u =\mu \partial_x\bigl(u^5\bigr) \label{eq:gKdV} \\
u(0,x)=u_0(x)\in L_x^2(\R), \notag
\end{gather}
where $\mu=\pm 1$ and the solution $u$ is a real-valued function of spacetime $(t,x)\in \R\times\R$.  When $\mu=1$ the equation is called
\emph{defocusing}, while the case $\mu=-1$ is known as \emph{focusing}.

The name \emph{mass-critical} refers to the fact that the scaling symmetry
\begin{equation}\label{scaling}
u(t,x) \mapsto  u^\lambda(t,x):= \lambda^{\frac12} u( \lambda^3 t, \lambda x)
\end{equation}
leaves both the equation \eqref{eq:gKdV} and the mass invariant.  The mass of a solution is defined by
\begin{equation}\label{mass}
M(u(t)) := \int_{\R} |u(t,x)|^2\, dx
\end{equation}
and is conserved under the flow.

Let us start by making the notion of a solution more precise.

\begin{definition}[Solution]\label{D:solution}
A function $u: I \times \R \to \R$ on a non-empty time interval $0\in I \subset \R$ is a \emph{(strong) solution} to \eqref{eq:gKdV} if it lies
in the class $C_t^0L^2_{x}(K\times\R)\cap L^5_x L_t^{10}(K\times \R)$ for all compact $K \subset I$, and obeys the Duhamel formula
\begin{equation}\label{eq:Duhamel}
u(t)=e^{-t\partial_x^3}u_0 + \mu \int_{0}^t e^{-(t-s)\partial_x^3}\partial_x\bigl(u^5(s)\bigr)\,ds
\end{equation}
for all $t\in I$.  We refer to the interval $I$ as the \emph{lifespan} of $u$. We say that $u$ is a \emph{maximal-lifespan
solution} if the solution cannot be extended to any strictly larger interval. We say that $u$ is a \emph{global solution} if $I
= \R$.
\end{definition}

Throughout this paper we will use the following notation:
$$
S_I(u):= \int_{\R} \Bigl(\int_I |u(t,x)|^{10} \,dt\Bigr)^{1/2}\, dx=\|u\|_{L_x^5L_t^{10}(I\times\R)}^5.
$$
In view of Theorem~\ref{T:local} below, we will also refer to $S_I(u)$ as the \emph{scattering size} of $u$ on the interval $I$.

Associated to the notion of solution is a corresponding notion of blowup, which we now define.  As we will see in Theorem~\ref{T:local}, this
precisely corresponds to the impossibility of continuing the solution (in the case of blowup in finite time) or failure to scatter
(in the case of blowup in infinite time).

\begin{definition}[Blowup]\label{D:blowup}
We say that a solution $u$ to \eqref{eq:gKdV} \emph{blows up forward in time} if there exists a time $t_1 \in I$ such that
$$ S_{[t_1, \sup I)}(u) = \infty$$
and that $u$ \emph{blows up backward in time} if there exists a time $t_1 \in I$ such that
$$ S_{(\inf I, t_1]}(u) = \infty.$$
\end{definition}

The local well-posedness theory for \eqref{eq:gKdV} with finite-mass initial data was developed by Kenig, Ponce, and Vega,
\cite{KPV}.  They constructed local-in-time solutions for arbitrary initial data in $L_x^2$; however, as is the case
with critical equations, the interval of time for which existence was proved depends on the profile of the initial data rather than on its
norm. Moreover, they constructed global-in-time solutions for small initial data in $L_x^2$ and showed that these solutions
scatter, that is, they are well approximated by solutions to the free Airy equation
$$
(\partial_t +\partial_{xxx} )u=0
$$
asymptotically in the future and in the past.  We summarize these results in the following theorem.

\begin{theorem}[Local well-posedness, \cite{KPV}]\label{T:local}
Given $u_0 \in L^2_x(\R)$ and $t_0 \in \R$, there exists a unique maximal-lifespan solution $u$ to \eqref{eq:gKdV} with
$u(t_0)=u_0$. We will write $I$ for the maximal lifespan.  This solution also has the following properties:
\begin{CI}
\item {\rm(}Local existence{\rm)} $I$ is an open neighbourhood of $t_0$.
\item {\rm(}Blowup criterion{\rm)} If $\sup I$ is finite, then $u$ blows up forward in time; if $\inf I$ is finite,
then $u$ blows up backward in time.
\item {\rm(}Scattering{\rm)} If $\sup I=+\infty$ and $u$ does not blow up forward in time, then $u$ scatters forward in time, that is,
there exists a unique $u_+ \in L^2_x(\R)$ such that
\begin{align}\label{like u+}
\lim_{t \to +\infty} \| u(t)-e^{-t\partial_x^3} u_+ \|_{L^2_x(\R)} = 0.
\end{align}
Conversely, given $u_+ \in L^2_x(\R)$ there is a unique solution to \eqref{eq:gKdV} in a neighbourhood of infinity so that
\eqref{like u+} holds.
\item {\rm(}Small data global existence{\rm)} If $M(u_0)$ is sufficiently small, then $u$ is a global solution
which does not blow up either forward or backward in time.  Indeed, in this case
$$S_\R(u)\lesssim M(u)^{5/2}.$$
\end{CI}
\end{theorem}

Global well-posedness for large finite-mass initial data is an open question.  In the case of more regular initial data, for
example, $u_0\in H^s_x(\R)$ with $s\geq 1$, one may access higher regularity conservation laws to answer the global
well-posedness question. One such conserved quantity is the energy,
\begin{equation}\label{energy}
E(u(t)) := \int_{\R} \tfrac12|\partial_x u(t,x)|^2 + \tfrac{\mu}6 |u(t,x)|^6\, dx.
\end{equation}
Invoking the conservation of energy, in the defocusing case one may iterate the local well-posedness theory to obtain a global
solution for initial data $u_0\in H^s_x(\R)$ with $s\geq 1$, \cite{KPV}.  In the focusing case, the same argument combined with
the sharp Gagliardo--Nirenberg inequality, \cite{Weinstein}, yields global well-posedness for finite-energy initial data with
mass less than that of the ground state soliton, which we will discuss in a moment.  In neither case does the argument yield information
about the long-time behaviour of the solution.

The ground state is the unique positive radial solution to the elliptic equation
$$
\partial_{xx}Q + Q^5 =Q
$$
and is given by the explicit formula
\begin{equation}\label{eq:groud-state}
Q(x)=\frac {3^{1/4}}{\cosh^{1/2}(2x)}.
\end{equation}
The ground state plays an important role in the study of the focusing case ($\mu=-1$) of \eqref{eq:gKdV}, as it gives rise to \emph{soliton} solutions.
More precisely,
$$
u(x,t):= c^{\frac 14}Q\bigl(\sqrt{c}(x-ct)\bigr) \quad \text{for} \quad c>0
$$
is a solution to \eqref{eq:gKdV}.  Furthermore, it is known that when $M(u_0)>M(Q)$, solutions can blow up in finite time, \cite{Merle:gKdV}, even for
$H^1_x$ initial data.

There has been some work dedicated to lowering the regularity of the initial data for which one has global well-posedness. In
\cite{FLP}, Fonseca, Linares, and Ponce established global well-posedness for solutions of the focusing mass-critical gKdV for
initial data in $H^s(\R)$ with $s>3/4$ and mass less than that of the ground state solution.  Recently, Farah, \cite{Farah},
used the I-method of Colliander, Keel, Staffilani, Takaoka, and Tao, \cite{CKSTT:almost}, to further lower the regularity of the
initial data to $s>3/5$.  In view of the fact that it is both scaling-critical
and conserved by the flow, it is natural to endeavour to prove well-posedness for initial data in $L_x^2$, that is, when $s=0$.

Another interesting open question is related to the asymptotic behavior of global solutions to \eqref{eq:gKdV}.  Intuitively, if
we knew that $u(t)$ decayed to zero (in some sense) as $t\to \pm \infty$, then the nonlinearity $\partial_x(u^5(t))$
should decay even faster and so the nonlinear effects should become negligible for large times.  As a result, it is widely
expected that the nonlinear solution scatters to a linear solution, at least in the defocusing case; in the focusing case, the
same behavior is expected for initial data with mass less than that of the ground state.  More precisely, it is expected that
there exist $u_\pm\in L_x^2(\R)$ such that
\begin{equation}\label{eq:asymptotic completeness}
\bigl\|u(t) -e^{-t\partial_x^3}u_{\pm}\bigr\|_{L_x^2} \to 0 \quad \text{as} \quad t \to \pm\infty.
\end{equation}

For critical problems it is natural to encapsulate both the well-posedness and scattering questions in the form of global
spacetime bounds; the precise formulation is Conjecture~\ref{Conj:gKdV} below.  Indeed, the existence of a scaling symmetry
implies that there is no reference scale for time or space and hence, one should regard `good' (i.e., profile-independent) local
well-posedness and scattering as two facets of the same question.  In addition to addressing global well-posedness and scattering,
spacetime bounds imply a strong form of stability for the equation; see Theorem~\ref{T:stab}.

\begin{conjecture}[Spacetime bounds for the mass-critical gKdV]\label{Conj:gKdV}
\hspace*{1em}The defocusing mass-critical gKdV is globally well-posed for arbitrary initial data $u_0\in L^2_x(\R)$.  In the focusing case,
the same conclusion holds for initial data $u_0\in L^2_x(\R)$ with $M(u_0)<M(Q)$.  Furthermore, in both cases, the global
solution satisfies the following spacetime bounds:
\begin{align}\label{S bound}
\|u\|_{L_x^5L_t^{10}(\R\times\R)}\leq C(M(u_0)).
\end{align}
\end{conjecture}

Conjecture~\ref{Conj:gKdV} has been compared in the literature to the analogous conjecture for the mass-critical nonlinear
Schr\"odinger equation in one space dimension. This is
\begin{equation}\label{eq:NLS}
\begin{cases}
-iv_t+\partial_{xx} v=\frac 5{24} \mu |v|^4v\\
v(0,x)=v_0(x)\in L_x^2(\R),
\end{cases}
\end{equation}
where $\mu=\pm 1$ and the solution $v$ is a complex-valued function of spacetime $\R\times\R$.  Just as for the mass-critical
gKdV, the case $\mu=1$ is called defocusing, while the case $\mu=-1$ is known as focusing.  The numerical constant $\frac5{24}$
can be changed to any other positive value by rescaling $v$.  However, as will be discussed below, this specific value is convenient
for exhibiting a close connection between \eqref{eq:gKdV} and \eqref{eq:NLS}.  More precisely, it was observed in \cite{CCT, Tao:2remarks}
that for highly oscillatory initial data, solutions of gKdV mimic those of NLS.

Note also that \eqref{eq:NLS} is time-reversed relative to most work on this equation; positive frequencies move to the left.

Mass and energy as defined by \eqref{mass} and \eqref{energy} are also conserved quantities for \eqref{eq:NLS}.  Moreover,
\eqref{eq:NLS} enjoys a scaling symmetry
\begin{equation}\label{scaling-NLS}
v(t,x) \mapsto  v^\lambda(t,x):= \lambda^{\frac12} v( \lambda^2 t, \lambda x),
\end{equation}
which leaves both the equation \eqref{eq:NLS} and the mass invariant.

In the focusing case, \eqref{eq:NLS} admits soliton solutions.  More precisely, for $\mu=-1$,
\begin{equation}\label{eq:re-scaled-soliton}
v(t,x):= e^{-i \frac 5{24}t}Q\bigl(\sqrt{\tfrac 5{24}} \, x\bigr)
\end{equation}
is a solution to \eqref{eq:NLS}, where $Q$ is as defined in \eqref{eq:groud-state}.  Note that $M(v)=2\rootsixfifths M(Q)$.

The local theory for \eqref{eq:NLS} was developed by Cazenave and Weissler, \cite{cw0, cwI}, who constructed local-in-time
solutions for arbitrary initial data in $L_x^2$ (with the time of existence depending on the profile of the initial data)
and global-in-time solutions for small initial data in $L_x^2$.

For finite-energy initial data, the usual iterative argument yields global existence in the defocusing case.  In the focusing case,
global existence also follows from the same argument  combined with the sharp Gagliardo--Nirenberg inequality for finite-energy
initial data with $M(v_0)<2\rootsixfifths M(Q)$; see \cite{Weinstein}.  For global existence results for less regular data, but
still above the critical regularity, see \cite{PSST:1D, Tzirakis}.

The natural global well-posedness and scattering conjecture for \eqref{eq:NLS} is the following; it is still open.

\begin{conjecture}[Spacetime bounds for the mass-critical NLS]\label{Conj:NLS}
The defocusing mass-critical NLS is globally well-posed for arbitrary initial data $v_0\in L^2_x(\R)$.  In the focusing case,
the same conclusion holds for initial data $v_0\in L^2_x(\R)$ with $M(v_0)<2\rootsixfifths M(Q)$.  Furthermore, in both cases,
the global solution satisfies the following spacetime bounds:
\begin{align*}
\int_\R\int_\R |v(t,x)|^6 \,dx\,dt \leq C(M(v_0)).
\end{align*}
\end{conjecture}

Recently, Tao \cite{Tao:2remarks} used the fact that solutions to \eqref{eq:NLS} can be used to build solutions to \eqref{eq:gKdV}
in order to show that Conjecture~\ref{Conj:gKdV} implies Conjecture~\ref{Conj:NLS}.  More precisely, he showed

\begin{theorem}[Conjecture~\ref{Conj:gKdV} almost implies Conjecture~\ref{Conj:NLS}, \cite{Tao:2remarks}]\label{T:Tao}
Fix $\mu=\pm 1$ and assume that Conjecture~\ref{Conj:gKdV} holds for initial data $u_0\in L_x^2(\R)$ with $M(u_0)< M$ for some
$M\in (0,\infty)$. Then Conjecture~\ref{Conj:NLS} holds for initial data $v_0\in L_x^2(\R)$ with $M(v_0)<2M$.
\end{theorem}

\begin{remark}
Note that in the defocusing case, Theorem~\ref{T:Tao} shows that the full version of Conjecture~\ref{Conj:gKdV} implies the full
version of Conjecture~\ref{Conj:NLS}.  In the focusing case, the result is somewhat inefficient as it only proves that the full
version of Conjecture~\ref{Conj:gKdV} implies Conjecture~\ref{Conj:NLS} in the small mass case $M(v_0)<2M(Q)$, missing the desired hypothesis
by a factor of $\sqrt{6/5}$.
\end{remark}

As this theorem shows, any attack on Conjecture~\ref{Conj:gKdV} must also address Conjecture~\ref{Conj:NLS}, at least in some way.
The approach we adopt here is to prove a form of converse to Theorem~\ref{T:Tao}, namely Theorem~\ref{T:NLS emb}, which we contend
(cf. Remark~\ref{R:inverse}) precisely isolates the role of NLS as an obstruction to proving Conjecture~\ref{Conj:gKdV}.

The principal thrust of this paper however, is to provide what we believe to be an important first step to verifying Conjecture~\ref{Conj:gKdV}.
Our reason for such optimism stems from recent progress on other dispersive equations at critical regularity: NLW, wave maps,
and more specifically, NLS.

The recent progress on the mass- and energy-critical NLS can be found in \cite{borg:scatter, CKSTT:gwp, grillakis, KenigMerle,
KTV, Berbec, KVZ, RV, Tao: gwp radial, tvz:cc, TVZ:sloth, V:thesis art}.  Here we apply the
techniques developed to attack these problems to Conjecture~\ref{Conj:gKdV}.  More precisely, using concentration-compactness
techniques, we will show that if Conjecture~\ref{Conj:gKdV} were to fail (but Conjecture~\ref{Conj:NLS} holds true), then there
exists a minimal-mass blowup solution to \eqref{eq:gKdV}.  Moreover, this minimal-mass counterexample to
Conjecture~\ref{Conj:gKdV} has good compactness properties.

To state our results, we need the following definition.

\begin{definition}[Almost periodicity modulo symmetries]\label{D:AP}
Fix $\mu=\pm 1$.  A solution $u$ to \eqref{eq:gKdV} with lifespan $I$ is said to be \emph{almost periodic modulo symmetries} if
there exist functions $N: I \to \R^+$, $x:I\to \R$, and $C: \R^+ \to \R^+$ such that for all $t \in I$ and $\eta > 0$,
$$ \int_{|x-x(t)| \geq C(\eta)/N(t)} |u(t,x)|^2\, dx
    +\int_{|\xi| \geq C(\eta) N(t)} | \hat u(t,\xi)|^2\, d\xi\leq \eta.
$$
We refer to the function $N$ as the \emph{frequency scale function} for the solution $u$, $x$ the \emph{spatial center
function}, and to $C$ as the \emph{compactness modulus function}.
\end{definition}

\begin{remark}
The parameter $N(t)$ measures the frequency scale of the solution at time $t$, while $1/N(t)$ measures the spatial scale. It is
possible to multiply $N(t)$ by any function of $t$ that is bounded both above and below, provided that we also modify the
compactness modulus function $C$ accordingly.
\end{remark}

\begin{remark}
It follows from the Ascoli--Arzela Theorem that a family of functions is precompact in $L^2_x(\R)$ if and only if it is norm-bounded and there
exists a compactness modulus function $C$ so that
$$
\int_{|x| \geq C(\eta)} |f(x)|^2\ dx + \int_{|\xi| \geq C(\eta)} |\hat f(\xi)|^2\ d\xi \leq \eta
$$
for all functions $f$ in the family.  Thus, an equivalent formulation of Definition~\ref{D:AP} is as follows: $u$ is almost
periodic modulo symmetries if and only if
$$
\{ u(t): t \in I \} \subseteq \{ \lambda^{\frac12} f(\lambda (x+x_0)) : \, \lambda\in(0,\infty), \ x_0\in \R, \text{ and }f \in
K \}
$$
for some compact subset $K$ of $L^2_x(\R)$.
\end{remark}

In view of the small data result in Theorem~\ref{T:local} and the stability result Theorem~\ref{T:stab}, the failure of
Conjecture~\ref{Conj:gKdV} is equivalent to the existence of a critical mass $M_c\in (0,\infty)$ when $\mu=1$ or $M_c\in
(0,M(Q))$ when $\mu=-1$ such that
\begin{equation}\label{eq:critical-mass}
L(M)<\infty \text{ for } M<M_c \quad \text{and} \quad L(M)=\infty \text{ for } M\geq M_c;
\end{equation}
here,
$$L(M):=\sup \{S_I(u):\, u: I\times \R\to\R \text{ with } M(u)\le M \}$$
and the supremum is taken over all solutions $u:I\times \R\to\R$ to \eqref{eq:gKdV} obeying $M(u)\leq M$.  Indeed,
for sufficiently small masses $M$, Theorem~\ref{T:local} gives
$$
L(M)\lesssim M^{5/2},
$$
while Theorem~\ref{T:stab} shows that $L$ is a continuous map into $[0,\infty]$ in both the defocusing and focusing cases.

Now we are ready to state the first result of this paper.

\begin{theorem}[Reduction to almost periodic solutions]\label{T:reduct}
Fix $\mu=\pm 1$ and assume Conjecture~\ref{Conj:gKdV} fails for this value of $\mu$.  Let $M_c$ denote the corresponding critical mass
and assume that Conjecture~\ref{Conj:NLS} holds for initial data with mass $M(v_0)\leq 2M_c$.  Then there exists a maximal-lifespan
solution $u$ to the mass-critical gKdV \eqref{eq:gKdV} with mass $M(u)=M_c$, which is almost periodic modulo symmetries and blows up
both forward and backward in time.
\end{theorem}

\begin{remark}
From the definition of $M_c$ and Theorem~\ref{T:Tao}, we see that Conjecture~\ref{Conj:NLS} holds for solutions $v$ to \eqref{eq:NLS} with
initial data $v_0\in L_x^2(\R)$ satisfying $M(v_0)< 2M_c$.  Thus, the hypothesis needed for Theorem~\ref{T:reduct} is that
Conjecture~\ref{Conj:NLS} holds for solutions to \eqref{eq:NLS} with $M(v_0)=2M_c$.
\end{remark}

Thus, in order to prove Conjecture~\ref{Conj:gKdV} (assuming that Conjecture~\ref{Conj:NLS} holds) it suffices to preclude the existence
of minimal-mass blowup solutions.  Employing a combinatorial argument in the spirit of \cite[Theorem~1.16]{KTV}, one can prove that no matter
how small the class of minimal-mass blowup solutions to \eqref{eq:gKdV} is, one can always find at least one of three specific enemies to
Conjecture~\ref{Conj:gKdV}.  More precisely, in Section~\ref{S:enemies}, we adapt the argument given in \cite{KTV} to obtain

\begin{theorem}[Three special scenarios for blowup]\label{T:enemies}
Fix $\mu=\pm 1$ and suppose that Conjecture~\ref{Conj:gKdV} fails for this choice of $\mu$.  Let $M_c$ denote the corresponding critical
mass and assume that Conjecture~\ref{Conj:NLS} holds for initial data with mass $M(v_0)\leq 2M_c$.  Then there exists a
maximal-lifespan solution $u:I\times\R\to \R$ to \eqref{eq:gKdV} with mass $M(u)=M_c$, which is almost periodic modulo symmetries and blows up both forward
and backward in time. Moreover, we can also ensure that the lifespan $I$ and the
frequency scale function $N:I\to\R^+$ match one of the following three scenarios:
\begin{itemize}
\item[I.] (Soliton-like solution) We have $I = \R$ and
\begin{equation*}
 N(t) = 1 \quad \text{for all} \quad t\in \R.
\end{equation*}
\item[II.] (Double high-to-low frequency cascade) We have $I = \R$,
\begin{equation*}
\liminf_{t \to -\infty} N(t) = \liminf_{t \to +\infty} N(t) = 0,
\end{equation*}
and
\begin{equation*}
\sup_{t \in \R} N(t) < \infty.
\end{equation*}
\item[III.] (Self-similar solution) We have $I = (0,+\infty)$ and
\begin{equation*}
N(t) = t^{-1/3} \quad \text{for all} \quad t\in I.
\end{equation*}
\end{itemize}
\end{theorem}

\begin{remark}\label{R:inverse}
In none of the three scenarios just described is there any known connection to solutions of NLS nor any other simpler equation.  This is our
justification for the claim made earlier that Theorem~\ref{T:NLS emb} precisely isolates the role of NLS in Conjecture~\ref{Conj:gKdV}.
\end{remark}

Using the analogue of Theorem~\ref{T:enemies} developed in the context of the mass-critical NLS (see \cite{KTV}), it is possible
to recast the role of Conjecture~\ref{Conj:NLS} in Theorems~\ref{T:reduct} and~\ref{T:enemies} as follows: Suppose Conjecture~\ref{Conj:gKdV}
fails; then there exists a minimal-mass blowup solution to \emph{either} \eqref{eq:gKdV} \emph{or} \eqref{eq:NLS}.  Moreover, in the former case
this solution can be taken to have the structure of one of the three scenarios listed above.  In the latter case the three scenarios are very similar
(cf. \cite{KTV}); there is an additional Galilei symmetry and the self-similar solution has $N(t)=t^{-1/2}$.

Let us now outline the content of the remainder of the paper with a few remarks on what novelties appear in the analysis.

In Section~\ref{S:lemmas} we recall the linear estimates that are essential for our arguments.  In particular, we recall
the linear profile decomposition for the Airy equation developed in \cite{Shao:Airy}.  Note that the defect of compactness
arising from highly oscillatory data (cf. the parameters $\xi_n$ in Lemma~\ref{L:profile decomp}) is not associated with a
symmetry of our equation; by comparison, in the mass-critical NLS context, it is associated to the Galilei boost symmetry.
This is the primary source of difficulty/novelty in our analysis; it is also the regime in which the connection to
Conjecture~\ref{Conj:NLS} arises.  An early manifestation of this nuisance appears when proving decoupling of the nonlinear profiles;
see Lemma~\ref{L:decoupling}.

In Section~\ref{S:stab} we develop a stability theory for gKdV, which controls the effect of both small perturbations to the initial data
and the addition of weak forcing terms.

In Section~\ref{S:NLS emb} we discuss the behaviour of highly oscillatory solutions to gKdV.  More precisely, we show how Conjecture~\ref{Conj:NLS}
implies the existence of spacetime bounds for such solutions to gKdV.  This is Theorem~\ref{T:NLS emb} and is our converse to Theorem~\ref{T:Tao}.
The relation between the proofs of these theorems will be elaborated upon there.

Section \ref{S:reduct} is devoted to the proof of Theorem~\ref{T:reduct}.  Again, the principal differences when compared to NLS arise
in the case of highly oscillatory profiles.  In particular, we rely upon Lemma~\ref{L:decoupling} and Theorem~\ref{T:NLS emb}.

The proof of Theorem~\ref{T:enemies} appears in Section~\ref{S:enemies} and is closely modeled on the analogous reduction for NLS
proved in \cite{KTV}.

\subsection*{Acknowledgements}  We would like to thank Terry Tao for useful comments.  The first author was supported by NSF grant DMS-0701085.
The last author was supported by NSF grant DMS-0901166.

%
%
%
%

\section{Notation and useful lemmas}\label{S:lemmas}

\subsection{Some notation}
We write $X \lesssim Y$ or $Y \gtrsim X$ to indicate that $X \leq CY$ for some constant $C>0$.  We use $O(Y)$ to denote any quantity $X$
such that $|X| \lesssim Y$.  We use the notation $X \sim Y$ to mean $X \lesssim Y \lesssim X$.  If $C$ depends upon some
parameters, we will indicate this with subscripts; for example, $X \lesssim_u Y$ denotes the assertion that $X \leq C_u Y$ for
some $C_u$ depending on $u$; similarly for $X \sim_u Y$, $X = O_u(Y)$, etc.

For any spacetime slab $I\times \R$, we use $L_t^qL_x^r(I\times \R)$ and $L_x^rL_t^q(I\times \R)$ respectively, to denote the
Banach spaces of functions $u: I\times \R\to \C$ whose norms are
$$
\|u\|_{L_t^qL_x^r(I\times\R)}:=\Bigl(\int_I\|u(t)\|_{L^r_x}^q \, dt\Bigr)^{1/q}
$$
and
$$
\|u\|_{L_x^rL_t^q(I\times\R)}:=\Bigl(\int_\R\|u(x)\|_{L_t^q (I)}^r \, dx\Bigr)^{1/r},
$$
with the usual modifications when $q$ or $r$ is equal to infinity.  When $q=r$ we abbreviate $L^q_t L^q_x$ and $L^q_x L^q_t$ as
$L^q_{t,x}$.

We define the Fourier transform on $\R$ by
$$
\hat f(\xi):= (2\pi)^{-1/2} \int_{\R} e^{-ix\xi}f(x)\,dx.
$$
For $s\in \R$, we define the fractional differentiation/integral operator
$$
\widehat{|\partial_x|^s f}(\xi):=|\xi|^s\hat f(\xi),
$$
which in turn defines the homogeneous Sobolev norms
$$
\|f\|_{\dot W_x^{s,r}(\R)}:=\bigl\||\partial_x|^s f\bigr\|_{L_x^r(\R)}.
$$
When $r=2$, we denote the space $\dot W_x^{s,r}(\R)$ by $\dot H^s_x(\R)$.

\subsection{Linear estimates}

We start by recalling the usual Kato smoothing, Strichartz, and maximal function estimates associated to the Airy propagator.

\begin{lemma}[Linear estimates, \cite{Kato:smoothing, KPV:1, KPV}]\label{L:linear estimates}
Let $I$ be a compact time interval and let $u:I\times \R\to \R$ be a solution to the forced Airy equation
\begin{equation*}
(\partial_t+\partial_{xxx})u=\partial_x F +G.
\end{equation*}
Then we have the Kato smoothing, maximal function, and Strichartz estimates
\begin{align*}
\bigl\|\partial_xu\bigr\|_{L^\infty_xL^2_t(I\times \R)}& + \bigl\||\partial_x|^{-1/4}u\bigr\|_{L^4_xL^\infty_t(I\times \R)}
    + \|u\|_{L_t^\infty L_x^2(I\times\R)}\\
& +\bigl\||\partial_x|^{1/6}u\bigr\|_{L^6_{t,x}(I\times \R)} + \|u\|_{L_x^5L_t^{10}(I\times\R)}\\
&\qquad\qquad \lesssim \|u(t_0)\|_{L^2(\R)} + \|F\|_{L^1_xL^2_t(I\times \R)} + \|G\|_{L^1_tL^2_x(I\times \R)}
\end{align*}
for any $t_0\in I$.
\end{lemma}

\subsection{A linear profile decomposition}
In this subsection we record the linear profile decomposition statement from \cite{Shao:Airy}, which will lead to the reduction in
Theorem~\ref{T:reduct}. For a linear profile decomposition for the Schr\"odinger propagator, see \cite{BegoutVargas,
carles-keraani, keraani-h1, keraani-l2, Notes, merle-vega, Shao:NLS}.

We first recall the (non-compact) symmetries of the equation \eqref{eq:gKdV} which fix the initial surface $t=0$.

\begin{definition}[Symmetry group]\label{D:sym}
For any position $x_0\in \R$ and scaling parameter $\lambda > 0$, we define a unitary transformation $g_{x_0,\lambda}: L_x^2(\R) \to L_x^2(\R)$ by
$$
[g_{x_0, \lambda} f](x) :=  \lambda^{-\frac12} f\bigl( \lambda^{-1}(x-x_0) \bigr).
$$
Let $G$ denote the collection of such transformations.  For a function $u: I \times \R \to \R$, we define
$T_{g_{x_0,\lambda}} u: \lambda^3 I \times \R \to \C$ where $\lambda^3 I := \{ \lambda^3 t: t \in I \}$ by the formula
$$
[T_{g_{x_0, \lambda}} u](t,x) :=  \lambda^{-\frac12} u\bigl( \lambda^{-3}t, \lambda^{-1}(x-x_0)\bigr).
$$
Note that if $u$ is a solution to \eqref{eq:gKdV}, then $T_{g}u$ is a solution to \eqref{eq:gKdV} with initial data $g u_0$.
\end{definition}

\begin{remark} It is easy to verify that $G$ is a group and that the map $g \mapsto T_g$ is a homomorphism.
Moreover, $u \mapsto T_g u$ maps solutions to \eqref{eq:gKdV} to solutions with the same Strichartz size as $u$, that is,
$$
\|T_{g_{x_0, \lambda}}(u)\|_{L_x^5L_t^{10}(\lambda^3 I\times\R)}= \|u\|_{L_x^5L_t^{10}(I\times\R)}
$$
and
$$
\bigl\||\partial_x|^{1/6}T_{g_{x_0, \lambda}}(u)\bigr\|_{L^6_{t,x}(\lambda^3 I\times \R)} = \bigl\||\partial_x|^{1/6}u\bigr\|_{L^6_{t,x}(I\times \R)}.
$$
Furthermore, $u$ is a maximal-lifespan solution if and only if $T_g u$ is also a maximal-lifespan solution.
\end{remark}

We are now ready to record the linear profile decomposition for the Airy propagator.

\begin{lemma}[Airy linear profile decomposition, \cite{Shao:Airy}]\label{L:profile decomp}
Let $\{u_n\}_{n\geq 1}$ be a sequence of real-valued functions bounded in $L^2_x(\R)$.  Then, after passing to a subsequence if
necessary, there exist (possibly complex) functions $\{\phi^j\}_{j\geq 1}\subset L^2_x(\R)$, group elements $\gnj \in G$, frequency parameters
$\xi_n^j\in [0,\infty)$, and times $\tnj\in \R$ such that for all $J\geq 1$ we have the decomposition
\begin{equation*}
u_n= \sum_{1\le j\le J} g_n^je^{-t_n^j\partial_x^3}[\Re(e^{ix\xi_n^j\lambda_n^j}\phi^{j})]+ w_n^J,
\end{equation*}
where the parameters $\xi_n^j$ satisfy the following property: for any $1\leq j\leq J$ either $\xi_n^j= 0$ for all $n\geq 1$, or
$\xi_n^j\lambda_n^j\to \infty$ as $n\to \infty$.  Here, $w^J_n \in L^2_x(\R)$ is real-valued and its linear evolution has
asymptotically vanishing symmetric Strichartz norm, that is,
\begin{equation}\label{error vanish 1}
\lim_{J\to \infty}\limsup_{n\to \infty} \bigl\||\partial_x|^{1/6}e^{-t\partial_x^3}w_n^J\bigr\|_{L^6_{t,x}(\R\times\R)}=0.
\end{equation}
Moreover, the following orthogonality conditions are satisfied:
\begin{itemize}
\item For any $j\neq k$,
\begin{align}\notag
\lim_{n\to\infty}\Biggl[\frac{\lambda_n^j}{\lambda_n^k} & + \frac{\lambda_n^k}{\lambda_n^j} +
+ \sqrt{\lambda_n^j \lambda_n^k}\,\bigl|\xi_n^j-\xi_n^k\bigr|
+ \langle\lambda_n^j\xi_n^j\lambda_n^k\xi_n^k\rangle^{1/2} \biggl|\frac{(\lambda_n^j)^3t_n^j-(\lambda_n^k)^3t_n^k}{\bigl(\lambda_n^j\lambda_n^k\bigr)^{3/2}}\biggr| \\
\label{orthog 1}
& + (\lambda_n^j\lambda_n^k\bigr)^{-1/2} \biggl|x_n^j-x_n^k + \tfrac32
        \bigl[(\lambda_n^j)^3t_n^j-(\lambda_n^k)^3t_n^k\bigr] \bigl[(\xi_n^j)^2+(\xi_n^k)^2\bigr]\biggr|\Biggr]=\infty.
\end{align}
\item For any $J\geq 1$,
\begin{align} \label{orthog 2}
\lim_{n\to\infty} \Biggl[\|u_n\|^2_{L^2_x}-\sum_{1\le j\le J}\bigl\|\Re[e^{ix\xi_n^j\lambda_n^j}\phi^j]\bigr\|^2_{L^2_x} -
\|w_n^J\|^2_{L^2_x}\Biggr]=0.
\end{align}
\end{itemize}
\end{lemma}

\begin{remark}\label{R:wnj small}
By analytic interpolation together with Lemma~\ref{L:linear estimates} and \eqref{error vanish 1}, we obtain that the linear evolution
of the error term $w_n^J$ also vanishes asymptotically in the Strichartz space $L_x^5 L_t^{10}$.  Indeed,
\begin{align*}
\lim_{J\to\infty}\limsup_{n\to\infty}\|&e^{-t\partial_x^3}w_n^J\|_{L^5_xL^{10}_t(\R \times \R)}\notag\\
&\lesssim \lim_{J\to\infty}\limsup_{n\to\infty}\bigl\||\partial_x|^{-1/4}e^{-t\partial_x^3}w_n^J\bigr\|_{L^4_xL^\infty_t(\R \times \R)}^{2/5}
    \bigl\||\partial_x|^{1/6}e^{-t\partial_x^3}w_n^J\bigr\|_{L_{t,x}^6(\R \times \R)}^{3/5} \notag\\
&\lesssim \|w_n^J\|_{L_x^2}^{2/5} \lim_{J\to\infty}\limsup_{n\to\infty}\bigl\||\partial_x|^{1/6}e^{-t\partial_x^3}w_n^J\bigr\|_{L_{t,x}^6(\R \times \R)}^{3/5}
=0.
\end{align*}
\end{remark}

Our next lemma shows that divergence of parameters in the sense of \eqref{orthog 1} gives decoupling of \emph{nonlinear} profiles.
Note that when $\xi_n\lambda_n\to\infty$, the structure of the nonlinear profile is dictated by Theorem~\ref{T:NLS emb}.

\begin{lemma}[Decoupling for the nonlinear profiles]\label{L:decoupling}
Let $\psi^j$ and $\psi^k$ be functions in $C^\infty_c(\R\times\R)$.  Given sequences of parameters that diverge in the sense of \eqref{orthog 1}, we have
\begin{align}\label{decouple:slow slow}
\Bigl\|T_{g^j_n}\bigl[\psi^j(t+t_n^j)\bigr] \, T_{g^k_n}\bigl[\psi^k(t+t_n^k)\bigr]\Bigr\|_{L^{5/2}_x L^5_t} \to 0
\end{align}
in the case $\xi_n^j\equiv\xi_n^k\equiv 0$, while
\begin{gather}\label{decouple:fast slow}
\Bigl\|T_{g^j_n}\bigl[\psi^j\bigl(3\lambda^j_n\xi_n^j(t+t_n^j),x+3(\lambda^j_n\xi_n^j)^2(t+t_n^j)\bigr)\bigr] \, T_{g^k_n}\bigl[\psi^k(t+t_n^k)\bigr]\Bigr\|_{L^{5/2}_x L^5_t} \to 0
\end{gather}
when $\xi_n^j\lambda_n^j\to\infty$ and $\xi_n^k\equiv 0$. Lastly,
\begin{align}\notag
\Bigl\|T_{g^j_n}\bigl[\psi^j\bigl(3\lambda^j_n\xi_n^j(t+t_n^j),x&+3(\lambda^j_n\xi_n^j)^2(t+t_n^j)\bigr)\bigr] \times \\
\label{decouple:fast fast}
    &T_{g^k_n}\bigl[\psi^k\bigl(3\lambda^k_n\xi_n^k(t+t_n^k),x+3(\lambda^k_n\xi_n^k)^2(t+t_n^k)\bigr)\bigr]\Bigr\|_{L^{5/2}_x L^5_t} \to 0
\end{align}
when $\xi_n^j\lambda_n^j\to\infty$ and $\xi_n^k\lambda_n^k\to\infty$.
\end{lemma}

\begin{proof}
By moving the scaling symmetry onto one of the profiles (i.e., by changing variables in each of the space and time integrals) one can quickly
obtain convergence to zero unless $\lambda^j_n\sim\lambda^k_n$.  In the case of a rapidly moving profile, one should note that for any
$\vartheta\in C^\infty_c(\R\times\R)$,
$$
\bigl\| \vartheta(3a_n t,x+3a_n^2 t) \bigr\|_{L^{5/2}_x L^5_t} \lesssim 1
$$
independent of the growth of $a_n$.  This follows from the fact that
\begin{equation}\label{E:interval len}
x\mapsto \bigl\| \vartheta(3a_n t,x+3a_n^2 t) \bigr\|_{L^5_t} \quad\text{is supported on an interval of length $O(1+|a_n|)$.}
\end{equation}

With $\lambda^j_n\sim\lambda^k_n$, equation \eqref{decouple:fast slow} follows very quickly; one merely writes down the rather lengthy formula
and utilizes the fact that $\xi_n^j\lambda_n^j\to\infty$.  In the case of \eqref{decouple:slow slow}, one then sees that divergence of the spatial
or temporal center parameters, in the sense of \eqref{orthog 1}, eventually separates the supports of the two profiles.  Further details
can be found in a number of prior publications, including \cite{keraani-l2,Notes}.

We now turn our attention to \eqref{decouple:fast fast}.  The general scheme mimics that for \eqref{decouple:slow slow}; however,
everything becomes extremely messy without one small trick.  For this reason, we work through a few of the details.

Bounding $\psi^j$ and $\psi^k$ by (multiples of) the characteristic function of a suitably large square, we see that
\begin{equation}
\bigl[ \, \text{LHS\eqref{decouple:fast fast}} \, \bigr]^{5/2}
    \label{E:preCS}
    \lesssim \int\biggl(\int (\lambda_n^j\lambda_n^k)^{-5/2} \chi_{R_n^j}(t,x) \chi_{R_n^k}(t,x)  \,dt\biggr)^{1/2} \, dx
\end{equation}
where $R_n^j$ is the parallelogram
$$
R_n^j = \bigl\{ (t,x) : \bigl|\xi_n^j[t+(\lambda_n^j)^3t_n^j]\bigr| \lesssim (\lambda_n^j)^2
    \text{ and } \bigl|x-x_n^j+3(\xi_n^j)^2[t+(\lambda_n^j)^3t_n^j]\bigr| \lesssim \lambda_n^j \bigr\}
$$
and similarly for $R_n^k$.

The next step is to apply the Cauchy--Schwarz inequality to the spatial integral in \eqref{E:preCS}.  Before doing so, let us gather some
information that will allow us to bound what results.  First, by changing variables according to
$$
    v=x-x_n^j+3(\xi_n^j)^2[t+(\lambda_n^j)^3t_n^j] \quad\text{and}\quad    w=x-x_n^k+3(\xi_n^k)^2[t+(\lambda_n^k)^3t_n^k],
$$
we see that
\begin{equation}\label{Area of intersection}
\text{Area}\bigl( R_n^j \cap R_n^k \bigr)
    \lesssim \int_{|v|\lesssim\lambda_n^j} \int_{|w|\lesssim\lambda_n^k} \frac{dw\,dv}{(\xi_n^j+\xi_n^k)|\xi_n^j-\xi_n^k|}
    \lesssim \frac{\lambda_n^j \lambda_n^k}{(\xi_n^j+\xi_n^k)|\xi_n^j-\xi_n^k|},
\end{equation}
where the denominator originates from the Jacobian factor.  On the other hand,
\begin{equation}\label{Width of intersection}
\bigl| \bigl\{ x  : \bigl\| \chi_{R_n^j}(t,x) \chi_{R_n^k}(t,x) \bigr\|_{L^{5}_t} \neq 0 \bigr\}\bigr|
    \lesssim \min\bigl\{(\lambda_n^j)^2\xi_n^j,(\lambda_n^k)^2\xi_n^k\},
\end{equation}
just as in \eqref{E:interval len}.  Thus, combining \eqref{E:preCS}, \eqref{Area of intersection}, and \eqref{Width of intersection} with the Cauchy-Schwarz
inequality and the fact that we may assume $\lambda_n^j\sim\lambda_n^k$ yields
\begin{equation}
\bigl[ \, \text{LHS\eqref{decouple:fast fast}} \, \bigr]^{10}
    \lesssim  \bigl( \lambda_n^j \lambda_n^k |\xi_n^j-\xi_n^k|^2 \bigr)^{-1}.
\end{equation}
This shows convergence to zero unless
\begin{equation}\label{E:xi difference}
\sqrt{\smash{\lambda_n^j \lambda_n^k} \vrule width 0mm depth 0mm height 1.85ex } \, \bigl|\xi_n^j-\xi_n^k\bigr| \lesssim  1
\end{equation}
and is the origin of the second term in \eqref{orthog 1}.

It is now not difficult to deal with the remaining two terms in \eqref{orthog 1}; however, it is useful to observe that
\eqref{E:xi difference} and $\lambda_n^j\xi_n^j\to\infty$ imply $\xi_n^j\sim\xi_n^k$.  Indeed, the ratio converges to one.
\end{proof}


\begin{lemma}[Schr\"odinger maximal function, \cite{KPV:1}]\label{L:Gulkan}
Given a solution $u:I\times\R\to \C$ to
$$
iu_t = u_{xx} + G
$$
we have
$$
\| u \|_{L^4_xL^\infty_t(I\times\R)} \lesssim \bigl\| |\partial_x|^{1/4} u(t_0) \bigr\|_{L^2_x(\R)} + \bigl\| |\partial_x|^{1/4} G \bigr\|_{L^1_tL^2_x(I\times\R)}
$$
for any $t_0\in I$.
\end{lemma}

\begin{proof}
When $G=0$, this can be proved by a simple $TT^*$ argument; however, the result seems to appear for the first time in \cite{KPV:1},
which considers a much more general setup.  We note that $G$ can be inserted a posteriori by a simple application of Minkowski's inequality.
\end{proof}

%
%
%
%

\section{Stability theory}\label{S:stab}

An important part of the local well-posedness theory is the stability theory.  By stability, we mean the following property:
Given an \emph{approximate} solution $\tilde u$ to \eqref{eq:gKdV} in the sense that
\begin{equation}\label{kdv approx}
\begin{cases}
(\partial_t + \partial_{xxx})\tilde u = \partial_x\bigl(\tilde u^5\bigr) + e\\
\tilde u(0,x)= \tilde u_0(x)
\end{cases}
\end{equation}
with $e$ small in a suitable sense and the initial data $\tilde u_0$ close to $u_0$, then there exists a
\emph{genuine} solution $u$ to \eqref{eq:gKdV} which stays very close to $\tilde u$ in critical spacetime norms.  The question of
continuous dependence of the solution upon the initial data corresponds to the case $e=0$.

Although stability is a local question, it has played an important role in all existing treatments of the global well-posedness
problem for the nonlinear Schr\"odinger equation at critical regularity.  It has also proved useful in the treatment
of local and global questions for more exotic nonlinearities \cite{Matador, xiaoyi:matador}.  As in previous work, the stability result is
an essential tool for extracting a minimal-mass blowup solution.

\begin{theorem}[Long-time stability for the mass-critical gKdV]\label{T:stab}
Let $I$ be a time interval containing zero and let $\tilde u$ be a solution to \eqref{kdv approx} on $I\times\R$ for some function $e$.  Assume that
$$
\|\tilde u\|_{L^\infty_tL^2_x(I\times\R)}\le M, \quad \|\tilde u\|_{L^5_xL_t^{10}(I\times\R)}\le L
$$
for some positive constants $M$ and $L$.  Let  $u_0$ be such that
$$
\|u_0-\tilde u_0\|_{L^2_x}\le M'
$$
for some positive constant $M'$.  Assume also the smallness conditions
\begin{align*}
\bigl\|e^{-t\partial_x^3}(u_0-\tilde u_0)\bigr\|_{L^5_xL_t^{10}(I\times \R)}&\le \eps\\
\bigl\||\partial_x|^{-1}e\bigr\|_{L^1_xL^2_t(I\times \R)} &\le \eps
\end{align*}
for some small $0<\eps<\eps_1=\eps_1(M,M',L)$.  Then there exists a solution $u$ to \eqref{eq:gKdV} on $I\times \R$ with initial
data $u_0$ at time $t=0$ satisfying
\begin{align*}
\|u-\tilde u\|_{L^5_xL_t^{10}(I\times\R)} &\leq C(M,M',L) \eps\\
\bigl\|u^5-\tilde u^5\bigr\|_{L^1_xL^2_t(I\times \R)} &\leq C(M,M',L) \eps\\
\|u-\tilde u\|_{L_t^\infty L_x^2(I\times\R)} + \bigl\||\partial_x|^{1/6}(u-\tilde u)\bigr\|_{L_{t,x}^6(I\times\R)}&\leq C(M,M',L)\\
\bigl\|\partial_x u\bigr\|_{L^\infty_xL^2_t(I\times \R)} + \bigl\||\partial_x|^{-1/4}u\bigr\|_{L^4_xL^\infty_t(I\times \R)}
+ \|u\|_{L_t^\infty L_x^2(I\times\R)} \\
+\, \bigl\||\partial_x|^{1/6}u\bigr\|_{L_{t,x}^6(I\times\R)}+ \|u\|_{L^5_xL_t^{10}(I\times\R)}&\leq C(M,M',L).
\end{align*}
\end{theorem}

\begin{remark}
Theorem~\ref{T:stab} implies the existence and uniqueness of maximal-lifespan solutions to \eqref{eq:gKdV}.  It also proves that
the solutions depend uniformly continuously on the initial data (on bounded sets) in spacetime norms which are critical with respect to
scaling.
\end{remark}

The proof of a stability result is by now standard; we follow the exposition in \cite{tvz:cc}.  One first obtains a short-time
stability result which can be iterated to obtain a long-time stability result, as long as the number of iterations depends only
on the mass and the Strichartz norm.

\begin{lemma}[Short-time stability]\label{L:stab 1}
Let $I$ be a time interval containing zero and let $\tilde u$ be a solution to \eqref{kdv approx} on $I\times\R$ for some function $e$.  Assume that
$$
\|\tilde u\|_{L^\infty_tL^2_x(I\times\R)}\le M
$$
for some positive constant $M$.  Let  $u_0$ be such that
$$
\|u_0-\tilde u_0\|_{L^2_x}\le M'
$$
for some positive constant $M'$.  Assume also the smallness conditions
\begin{align}
\|\tilde u\|_{L^5_xL^{10}_t(I\times \R)}&\le \eps_0 \label{eq:small 1}  \\
\bigl\|e^{-t\partial_x^3}(u_0-\tilde u_0)\bigr\|_{L^5_xL^{10}_t(I\times \R)}&\le \eps \label{eq:small 2}  \\
\bigl\||\partial_x|^{-1}e\bigr\|_{L^1_xL^2_t(I\times \R)} &\le \eps \label{eq:small 3}
\end{align}
for some small $0<\eps<\eps_0=\eps_0(M,M')$.  Then there exists a solution $u$ to \eqref{eq:gKdV} on $I\times \R$ with initial
data $u_0$ at time $t=0$ satisfying
\begin{align}
\|u-\tilde u\|_{L^5_xL^{10}_t(I\times \R)} &\lesssim \eps \label{concl 4}\\
\bigl\|u^5-\tilde u^5\bigr\|_{L^1_xL^2_t(I\times \R)} &\lesssim \eps \label{concl 5}\\
\|u-\tilde u\|_{L_t^\infty L_x^2(I\times\R)} + \bigl\||\partial_x|^{1/6}(u-\tilde u)\bigr\|_{L_{t,x}^6(I\times\R)}&\lesssim M' \label{concl 6}\\
\bigl\|\partial_x u\bigr\|_{L^\infty_xL^2_t(I\times \R)} + \bigl\||\partial_x|^{-1/4}u\bigr\|_{L^4_xL^\infty_t(I\times \R)}
+ \|u\|_{L_t^\infty L_x^2(I\times\R)}\notag\\
+ \, \bigl\||\partial_x|^{1/6}u\bigr\|_{L_{t,x}^6(I\times\R)}+ \|u\|_{L^5_xL^{10}_t(I\times \R)}&\lesssim M+M'. \label{concl 7}
\end{align}
\end{lemma}

\begin{proof}
By the local well-posedness theory, it suffices to prove \eqref{concl 4} through \eqref{concl 7} as \emph{a priori} estimates,
that is, we may assume that the solution $u$ already exists.  Also, we may assume, without loss of generality, that $0=\inf I$.

Let $w:=u-\tilde u$. Then $w$ satisfies the following initial-value problem
\begin{align*}
\begin{cases}
(\partial_t+\partial_{xxx})w = \partial_x\bigl((\tilde u + w)^5 -\tilde u^5 \bigr) -e \\
w(0)=u_0-\tilde u_0.
\end{cases}
\end{align*}
For $t \in I$, we write
\begin{align*}
A(t) := \|w\|_{L^5_xL^{10}_t([0,t]\times \R)}\quad \text{and}\quad B(t) := \bigl\|(\tilde u+w)^5-\tilde u^5\bigr\|_{L^1_xL^2_t([0,t]\times \R)}.
\end{align*}
Then by Lemma~\ref{L:linear estimates}, \eqref{eq:small 2}, and \eqref{eq:small 3} we estimate
\begin{align*}
A(t)
&\lesssim \|e^{-t\partial_x^3}w(0)\|_{L^5_xL^{10}_t([0,t]\times \R)} + B(t) + \bigl\||\partial_x|^{-1}e\bigr\|_{L^1_xL^2_t([0,t]\times \R)} \\
&\lesssim \eps + B(t).
\end{align*}
On the other hand, H\"older's inequality yields
\begin{align}\label{B(t)}
B(t)\lesssim \|w\|_{L^5_xL^{10}_t}\bigl( \|w\|_{L^5_xL^{10}_t} + \|\tilde u\|_{L^5_xL^{10}_t}\bigr)^4 \lesssim A(t)^5 + \eps_0^4 A(t),
\end{align}
where all spacetime norms are on $[0,t]\times\R$.  Thus, we obtain
\begin{align*}
A(t)&\lesssim \eps + A(t)^5 + \eps_0^4 A(t),
\end{align*}
from which a continuity argument yields
$$
A(t)\lesssim \eps \quad \text{for all} \quad  t\in I,
$$
provided $\eps_0$ is chosen sufficiently small.  This proves \eqref{concl 4}.  Conclusion \eqref{concl 5} follows from
\eqref{concl 4} and \eqref{B(t)}.

It remains to establish \eqref{concl 6} and \eqref{concl 7}.  By Lemma~\ref{L:linear estimates}, \eqref{concl 5}, and the hypotheses,
\begin{align*}
\|u-\tilde u\|_{L_t^\infty L_x^2(I\times\R)} + &\bigl\||\partial_x|^{1/6}(u-\tilde u)\bigr\|_{L_{t,x}^6(I\times\R)}\\
&\lesssim \|w(0)\|_{L_x^2} + \bigl\|(\tilde u+w)^5-\tilde u^5\bigr \|_{L^1_xL^{2}_t(I\times\R)} +\bigl\||\partial_x|^{-1}e\bigr\|_{L^1_xL^2_t(I\times \R)}\\
&\lesssim M' +\eps
\end{align*}
and
\begin{align*}
\text{LHS\eqref{concl 7}}
&\lesssim \|u_0\|_{L^2_x} + \|u^5\|_{L^1_xL^2_t(I\times \R)} \\
&\lesssim \|\tilde u_0\|_{L^2_x} + \|w(0)\|_{L^2_x} + \|\tilde u^5\|_{L^1_xL^2_t(I\times \R)}+ \bigl\|u^5-\tilde u^5\bigr \|_{L^1_xL^{2}_t(I\times\R)}\\
&\lesssim M+M' +\eps_0^5+\eps.
\end{align*}
Taking $\eps_0=\eps_0(M,M')$ sufficiently small, we derive \eqref{concl 6} and \eqref{concl 7}.

This completes the proof of the lemma.
\end{proof}

We are now ready to complete the proof of Theorem~\ref{T:stab}.

\begin{proof}[Proof of Theorem~\ref{T:stab}]
We will derive Theorem~\ref{T:stab} from Lemma~\ref{L:stab 1} by an iterative procedure.  First, we assume, without loss of
generality, that $0=\inf I$.  Now let $\eps_0=\eps_0(M,2M')$ be as in Lemma~\ref{L:stab 1}.  Note that we have to replace $M'$
by the slightly larger $2M'$ as the difference $u(t)-\tilde u(t)$ in $L_x^2$ may possibly grow in time.

Divide $I$ into $N$ many intervals $I_j=[t_j, t_{j+1}]$ such that on each time interval $I_j$ we have
\begin{align}\label{small 5-10}
\tfrac12 \eps_0\leq \|\tilde u\|_{L_x^5L_t^{10}(I_j\times\R)}< \eps_0.
\end{align}
We will first show that $N$ depends only on $\eps_0$ and $L$, and hence only on $M,M',L$.  Indeed, for $0\leq j<N-1$, let
$$
f_j(x):= \|\tilde u(x)\|_{L_t^{10}(I_j)}.
$$
Then, by the hypothesis,
\begin{align*}
\Bigl\|\Bigl(\sum_{j=0}^{N-1} |f_j\,|^{10}\Bigr)^{1/10}\Bigr\|_{L_x^5}= \|\tilde u\|_{L_x^5L_t^{10}(I\times\R)}\leq L.
\end{align*}
Also, using \eqref{small 5-10} and the triangle inequality, we obtain
\begin{align*}
\Bigl\|\sum_{j=0}^{N-1} |f_j\,|\Bigr\|_{L_x^5}\leq \sum_{j=0}^{N-1}\|\tilde u\|_{L_x^5L_t^{10}(I_j\times\R)} \sim N\eps_0
\end{align*}
and
\begin{align*}
\Bigl\|\Bigl(\sum_{j=0}^{N-1} |f_j\,|^5\Bigr)^{1/5}\Bigr\|_{L_x^5} = \Bigl\|\sum_{j=0}^{N-1} |f_j\,|^5\Bigr\|_{L_x^1}^{1/5} =
\Bigl(\sum_{j=0}^{N-1} \|\tilde u\|_{L_x^5L_t^{10}(I_j\times\R)}^5 \Bigr)^{1/5} \sim N^{1/5}\eps_0.
\end{align*}
As, by interpolation,
\begin{align*}
\Bigl\|\Bigl(\sum_{j=0}^{N-1} |f_j\,|^5\Bigr)^{1/5}\Bigr\|_{L_x^5} \lesssim \Bigl\|\Bigl(\sum_{j=0}^{N-1}
|f_j\,|^{10}\Bigr)^{1/10}\Bigr\|_{L_x^5}^{8/9} \Bigl\|\sum_{j=0}^{N-1} |f_j\,|\Bigr\|_{L_x^5}^{1/9},
\end{align*}
we obtain
$$
N^{1/5}\eps_0 \lesssim L^{8/9} (N\eps_0)^{1/9},
$$
which implies $N\lesssim \bigl(1+\frac{L}{\eps_0}\bigr)^{10}$.

Choosing $\eps_1$ sufficiently small depending on $N,M,M'$, we can apply Lemma~\ref{L:stab 1} inductively to obtain that for each
$0\leq j< N$ and all $0<\eps<\eps_1$,
\begin{equation}\label{bounds on j}
\left.
\begin{aligned}
\|u-\tilde u\|_{L_x^5L_t^{10}(I_j\times \R)} &\leq C(j,M, M') \eps \\
\bigl\|u^5-\tilde u^5\bigr\|_{L^1_xL^2_t(I_j\times \R)} &\leq C(j,M, M') \eps \\
\|u-\tilde u\|_{L_t^\infty L_x^2(I_j\times\R)} + \bigl\||\partial_x|^{1/6}(u-\tilde u)\bigr\|_{L_{t,x}^6(I_j\times\R)}&\leq C(j,M, M') \\
\bigl\|\partial_x u\bigr\|_{L^\infty_xL^2_t(I_j\times \R)} + \bigl\||\partial_x|^{-1/4}u\bigr\|_{L^4_xL^\infty_t(I_j\times \R)}
+ \|u\|_{L_t^\infty L_x^2(I_j\times\R)}\\
+ \, \bigl\||\partial_x|^{1/6}u\bigr\|_{L_{t,x}^6(I_j\times\R)}+ \|u\|_{L^5_xL^{10}_t(I_j\times \R)}&\leq C(j,M, M'),
\end{aligned}\quad
\right\}
\end{equation}
provided we can show that
\begin{align}
\|u(t_j)-\tilde u(t_j)\|_{L^2_x}&\le 2M' \label{claim 1}\\
\bigl\|e^{-(t-t_j)\partial_x^3}\bigl(u(t_j)-\tilde u(t_j)\bigr)\bigr\|_{L_x^5L_t^{10}(I_j\times \R)}&\leq C(j,M,M')\eps\leq \eps_0. \label{claim 2}
\end{align}

To derive \eqref{claim 1} we apply Lemma~\ref{L:linear estimates} and use the inductive hypotheses, namely, that \eqref{bounds on j} holds
for all preceding values of $j$:
\begin{align*}
\|u(t_j)-\tilde u(t_j)\|_{L^2_x}
&\leq \|u_0-\tilde u_0\|_{L^2_x} + \bigl\|u^5-\tilde u^5\bigr\|_{L^1_xL^2_t([0,t_j]\times \R)} + \bigl\| |\partial_x|^{-1}e\|_{L_x^1L_t^2([0,t_j]\times\R)}\\
&\leq M' + \sum_{k=0}^{j-1} C(k,M,M')\eps + \eps.
\end{align*}
Taking $\eps_1=\eps_1(N,M,M')$ sufficiently small yields \eqref{claim 1}.

To obtain \eqref{claim 2} we apply again Lemma~\ref{L:linear estimates} and use \eqref{bounds on j} to estimate
\begin{align*}
\bigl\| & e^{-(t-t_j)\partial_x^3} \bigl(u(t_j)-\tilde u(t_j)\bigr)\bigr\|_{L_x^5L_t^{10}(I_j\times \R)}\\
&\lesssim \bigl\|e^{-t\partial_x^3}(u_0-\tilde u_0)\bigr\|_{L_x^5L_t^{10}([0, t_j]\times \R)}
    + \bigl\|u^5-\tilde u^5\bigr\|_{L^1_xL^2_t([0,t_j]\times \R)} +\bigl\| |\partial_x|^{-1}e\|_{L_x^1L_t^2([0,t_j]\times\R)}\\
&\lesssim \eps + \sum_{k=0}^{j-1} C(k,M,M')\eps +\eps.
\end{align*}
This proves \eqref{claim 2} for $\eps_1=\eps_1(N,M,M')$ sufficiently small.

Summing the bounds in \eqref{bounds on j} over all subintervals $I_j$ completes the proof of the theorem.
\end{proof}

%
%
%
%

\section{Embedding NLS inside gKdV}\label{S:NLS emb}

The purpose of this section is to prove the following

\begin{theorem}[Oscillatory profiles]\label{T:NLS emb}
Assume that Conjecture~\ref{Conj:NLS} holds.  Let $\phi\in L_x^2$ be a complex-valued function; in the focusing case, assume also that
$M(\phi)<2\rootsixfifths M(Q)$.  Let $\{\xi_n\}_{n\geq 1}\subset (0,\infty)$ with $\xi_n \to \infty$ and let $\{t_n\}_{n\geq 1}\subset \R$ such that
$3\xi_nt_n$ converges to some $T_0\in [-\infty, \infty]$.  Then for $n$ sufficiently large there exists a global solution $u_n$ to \eqref{eq:gKdV}
with initial data at time $t=t_n$ given by
\begin{align}\label{E:initial data}
u_n(t_n,x)= e^{-t_n\partial_x^3} \Re (e^{ix\xi_n} \phi(x)).
\end{align}
Moreover, the solution obeys the global spacetime bounds
\begin{align}\label{STB}
\bigl\| |\partial_x|^{1/6} u_n\bigr\|_{L_{t,x}^6(\R\times\R)} + \|u_n\|_{L_x^5 L_t^{10}(\R\times\R)} \lesssim_\phi 1
\end{align}
and for every $\eps>0$ there exist $n_\eps\in \N$ and $\psi_\eps\in C_c^\infty(\R\times\R)$ so that
\begin{align}\label{smooth approx}
\bigl\| u_n(t,x) - \Re \bigl[e^{ix\xi_n+it\xi_n^3} \psi_\eps\bigl( 3\xi_nt , x+ 3\xi_n^2 t\bigr)\bigr]\bigr\|_{L_x^5L_t^{10}(\R\times\R)}\leq \eps,
\end{align}
for all $n\geq n_\eps$.
\end{theorem}

As noted in the introduction, this is a form of converse to Theorem~\ref{T:Tao}.  Let us briefly sketch the argument behind Theorem~\ref{T:Tao} as given
in \cite{Tao:2remarks}:  To prove Conjecture~\ref{Conj:NLS}, one merely needs to prove a priori spacetime bounds for Schwartz solutions to NLS on
a compact time interval.  As in \cite{CCT}, Tao exploits the fact that such solutions can be used to build approximate solutions to gKdV
of comparable size.  Conjecture~\ref{Conj:gKdV} controls the size of all solutions to gKdV and so also of these particular solutions.
Thus Conjecture~\ref{Conj:NLS} follows.

We have glossed over two subtleties in the argument.  First, the difference in scaling between NLS and gKdV means that they share no common critical
spacetime norm.  For this reason, the Schwartz nature of the solution and the compactness of the time interval play essential roles in Tao's
argument.  To prove Theorems~\ref{T:reduct} and~\ref{T:enemies}, we must contend with non-Schwartz solutions and work globally in time --- as
extremal objects, minimal-mass blowup solutions are not susceptible to a priori analysis.  Overcoming these difficulties represents the principal
novelty of this section.

The second subtlety stems from the necessity to use $X^{s,b}$-type estimates to control the discrepancy between the NLS and gKdV evolutions.
In this aspect, we borrow directly from \cite{Tao:2remarks}; see Lemma~\ref{L:error-control} below.

The remainder of this section is devoted to the

\begin{proof}[Proof of Theorem~\ref{T:NLS emb}]
For $|T_0|<\infty$, we define $v_n$ to be the global solution to \eqref{eq:NLS} with initial data
$$
v_n(T_0)= P_{|\xi|\leq \xi_n^{1/4}} e^{-iT_0\partial_x^2}\phi.
$$
When $T_0=\pm \infty$, we define $v_n$ to be the global solution to \eqref{eq:NLS} that obeys
$$
\bigl\|v_n(t) - P_{|\xi|\leq \xi_n^{1/4}} e^{-it\partial_x^2}\phi \bigr\|_{L_x^2}\to 0 \quad \text{as} \quad t\to T_0.
$$
(The fact that $\phi$ need not be Schwartz necessitates the introduction of a frequency cutoff.)
Such solutions exist since we assume Conjecture~\ref{Conj:NLS}.  Moreover, these solutions obey
\begin{align}\label{vn 66}
\|v_n\|_{L_{t,x}^6(\R\times\R)}\lesssim_\phi 1.
\end{align}
Combining this with standard persistence of regularity arguments (cf. Lemma~3.10 in \cite{Matador}) and the frequency localization
of the initial data, we deduce that
\begin{align}\label{reg for vn}
\bigl\| |\partial_x|^s v_n\bigr\|_{L_t^5L_x^{10}(\R\times\R)} + \bigl\| |\partial_x|^s v_n\bigr\|_{L_{t,x}^6(\R\times\R)}
+ \bigl\| |\partial_x|^s v_n\bigr\|_{L_t^\infty L_x^2(\R\times\R)} \lesssim_\phi \xi_n^{s/4},
\end{align}
for any $s\geq 0$.

By the perturbation theory for the mass-critical NLS, as worked out in \cite{Matador}, we also have
\begin{equation}\label{v_n to v}
v_n\to v \quad \text{in}\quad L_{t,x}^6(\R\times\R) \cap C_t^0 L_x^2(\R\times\R),
\end{equation}
where $v$ is the solution to \eqref{eq:NLS} with
\begin{equation*}
\begin{cases}
v(T_0) = e^{-iT_0\partial_x^2}\phi, \quad & \text{if } |T_0|<\infty\\
\lim_{t\to T_0}\bigl\|v(t) - e^{-it\partial_x^2}\phi \bigr\|_{L_x^2}= 0, \quad &\text{if } T_0=\pm\infty.
\end{cases}
\end{equation*}
This solution also exists, is global, and scatters by Conjecture~\ref{Conj:NLS}.  In particular, there exist $v_\pm \in L_x^2$ so that
$$
\bigl\|v(T) - e^{-iT\partial_x^2}v_+ \bigr\|_{L_x^2} + \bigl\|v(-T) - e^{iT\partial_x^2}v_- \bigr\|_{L_x^2} \to 0 \quad \text{as}\quad T\to \infty.
$$
(Note that if $T_0=\pm \infty$, then we can identify one scattering state, namely, $v_\pm =\phi$.)  Using this and \eqref{v_n to v},
we deduce that
\begin{equation}\label{vn small tail}
\lim_{T\to \infty} \lim_{n\to \infty}\bigl(\| v_n \|_{L_{t,x}^6(|t|>T)} + \bigl\|v_n(\pm T) - e^{\mp iT\partial_x^2}v_\pm \bigr\|_{L_x^2} \bigr)=0.
\end{equation}

Next we use $v_n$ to build an approximate solution to gKdV, namely,
\begin{equation}\label{E:tilde u defn}
\tilde u_n(t,x) = \begin{cases} \Re\bigl[e^{ix\xi_n+it\xi_n^3} v_n( 3\xi_nt , x+ 3\xi_n^2 t)\bigr], & \text{ when } |t| \leq \tfrac{T}{3\xi_n} \\[0.5ex]
\exp\bigl\{-\bigl(t-\tfrac{T}{3\xi_n}\bigr) \partial_x^3\bigr\} \tilde u_n(\tfrac{T}{3\xi_n}),      & \text{ when } \,t >  \tfrac{T}{3\xi_n} \\[0.5ex]
\exp\bigl\{-\bigl(t+\tfrac{T}{3\xi_n}\bigr) \partial_x^3\bigr\} \tilde u_n(-\tfrac{T}{3\xi_n}),      & \text{ when } \,t <  - \tfrac{T}{3\xi_n}.
\end{cases}
\end{equation}
Here $T$ is large an $n$-independent parameter that will be chosen in due course.

Our first task is to show that this is indeed almost a solution to gKdV.  We begin with the simpler large-time regime.  While the cubic
dispersion relation of Airy can be well approximated by a suitable quadratic polynomial (and hence Schr\"odinger) in a bounded frequency regime
(note the frequency localization and shift in \eqref{E:tilde u defn}), the minute differences are magnified over long time scales.
Thus, one cannot maintain the approximation by NLS over large time intervals.  The key observation to deal with this is that a positive-frequency solution
which is well-dispersed (i.e., resembles a scattered wave) for NLS is also well-dispersed for gKdV.  This is captured by the following lemma.

\begin{lemma}\label{L:dispersion helps}
Let $\phi\in L_x^2$ and let $\{\xi_n\}_{n\geq 1} \subset (0, \infty)$ such that $\lim_{n\to \infty} \xi_n = \infty$.  Then
$$
\lim_{T\to \infty} \lim_{n\to \infty}\bigl\| |\partial_x|^{1/6} e^{-t\partial_x^3}[ e^{ix\xi_n} e^{-iT\partial_x^2} \phi]\bigr\|_{L_{t,x}^6([0,\infty)\times\R)}
=0.
$$
\end{lemma}

\begin{proof}
By the Strichartz inequality, it suffices to prove the claim when $\phi$ is a Schwartz function with compact Fourier support.

A computation reveals that
\begin{align*}
|\partial_x|^{1/6} e^{-t\partial_x^3}[ e^{ix\xi_n} e^{-iT\partial_x^2} \phi] (x)
&= \bigl[ g_t \ast (e^{ix\xi_n}\phi) \bigr](x),
\end{align*}
where
$$
g_t (x) := (2\pi)^{-1/2} \int_{\R} e^{ix\xi}e^{it\xi^3 + iT(\xi-\xi_n)^2} |\xi|^{1/6} \chi_{\supp \hat\phi}(\xi-\xi_n)
$$
and $\chi_{\supp \hat\phi}$ denotes the characteristic function of the Fourier support of $\phi$.

Invoking the Van der Corput estimate \cite[Corollary, p.334]{Stein:large} and taking $n$ sufficiently large, we obtain
$$
\|g_t\|_{L_x^\infty}\lesssim_\phi \frac{\xi_n^{1/6}}{(T+ \xi_n t)^{1/2}}.
$$
Thus, for $n$ large,
$$
\bigl\| |\partial_x|^{1/6} e^{-t\partial_x^3}[ e^{ix\xi_n} e^{-iT\partial_x^2} \phi]\bigr\|_{L_x^\infty}
\lesssim_\phi \frac{\xi_n^{1/6}}{(T+ \xi_n t)^{1/2}} \|\phi\|_{L_x^1}.
$$

On the other hand, a direct computation shows
$$
\bigl\| |\partial_x|^{1/6} e^{-t\partial_x^3}[ e^{ix\xi_n} e^{-iT\partial_x^2} \phi]\bigr\|_{L_x^2}
\lesssim \xi_n^{1/6} \|\phi\|_{H_x^{1/6}}
$$
for $n$ sufficiently large.

Interpolating between the two bounds, we get
$$
\bigl\| |\partial_x|^{1/6} e^{-t\partial_x^3}[ e^{ix\xi_n} e^{-iT\partial_x^2} \phi]\bigr\|_{L_x^6}^6
\lesssim_{\phi} \frac{\xi_n}{(T+ \xi_n t)^2}.
$$
Finally, integrating with respect to time and letting $T\to \infty$ we derive the claim.
\end{proof}

The smallness of the linear evolution provided by Lemma~\ref{L:dispersion helps} carries over easily to the nonlinear evolution:

\begin{lemma}[Good approximation to gKdV -- large times]\label{L:good approx tail}
For $\tilde u_n$ as defined above, we have
\begin{align*}
\lim_{T\to \infty}\lim_{n\to \infty}
\bigl\| |\partial_x|^{-1}\bigl[ (\partial_t + \partial_{xxx})\tilde u_n - \mu \partial_x(\tilde u_n^5) \bigr]\bigr\|_{L^1_x L^2_t(|t| > \frac{T}{3\xi_n})}=0.
\end{align*}
\end{lemma}
\begin{proof}
By the definition of $\tilde u_n$,
\begin{align*}
\bigl\| |\partial_x|^{-1}\bigl[ (\partial_t + \partial_{xxx})\tilde u_n - \mu \partial_x(\tilde u_n^5) \bigr]\bigr\|_{L^1_x L^2_t(|t| > \frac{T}{3\xi_n})}
=  \bigl\| \tilde u_n \bigr\|_{L^5_x L^{10}_t(|t| > \frac{T}{3\xi_n})}^5.
\end{align*}
We will only consider the contribution from $t>\frac{T}{3\xi_n}$ to the right-hand side; negative values of~$t$ can be handled identically.
By analytic interpolation together with Lemma~\ref{L:linear estimates},
\begin{align*}
\bigl\| \tilde u_n \bigr\|_{L^5_x L^{10}_t( t > \frac{T}{3\xi_n})}^5
&\lesssim \bigl\| |\partial_x|^{1/6} \tilde u_n \bigr\|_{L^6_{t,x}(t > \frac{T}{3\xi_n})}^3
    \bigl\| |\partial_x|^{-1/4}\tilde u_n \bigr\|_{L^4_x L^{\infty}_t(t > \frac{T}{3\xi_n})}^2 \\
&\lesssim \bigl\| |\partial_x|^{1/6} e^{-t\partial_x^3} \tilde u_n(\tfrac{T}{3\xi_n}) \bigr\|_{L^6_{t,x}(t>0)}^3 \|\phi\|_{L^2_x}^2 \\
&\lesssim_\phi \bigl\| |\partial_x|^{1/6} e^{-t\partial_x^3} [e^{ix\xi_n} v_n( T )] \bigr\|_{L^6_{t,x}(t>0)}^3 \\
&\lesssim_\phi \bigl\|v_n( T) - e^{-iT\partial_x^2}v_+ \bigr\|_{L_x^2}^3 + \bigl\| |\partial_x|^{1/6} e^{-t\partial_x^3} [e^{ix\xi_n} e^{-iT\partial_x^2}v_+] \bigr\|_{L^6_{t,x}(t>0)}^3.
\end{align*}
Invoking \eqref{vn small tail} and Lemma~\ref{L:dispersion helps}, we derive the claim.
\end{proof}

We now turn to showing that $\tilde u_n$ is a good approximate solution in the middle interval $|t| \leq \frac{T}{3\xi_n}$.  Here we have
\begin{equation}\label{E:defn of E_n}
(\partial_t+\partial_{xxx})\tilde u_n = \mu\partial_x(\tilde u_n^5) + E_n,
\end{equation}
where $E_n:=E_n^1+E_n^2+E_n^3$ and the errors $E_n^j$ for $1\leq j\leq 3$ are given by
\begin{align*}
E_n^1 &:= \xi_n \sum_{k=3,5} C_{1,k} \Re\Bigl[ e^{ik\xi_n x+ik\xi_n^3t} \bigl(|v_n|^4v_n\bigr)\bigl(3\xi_n t, x+ 3\xi_n^2t\bigr)\Bigr]\\
E_n^2 &:= \sum_{k=1,3,5} C_{2,k} \Re\Bigl[ e^{ik\xi_nx+ik\xi_n^3t} \bigl(|v_n|^4v_n\bigr)_x\bigl(3\xi_nt, x+ 3\xi_n^2t\bigr)\Bigr]\\
E_n^3 &:= C_3 \Re\Bigl[ e^{i\xi_nx+i\xi_n^3t} (v_n)_{xxx}\bigl(3\xi_nt, x+ 3\xi_n^2t\bigr)\Bigr],
\end{align*}
with absolute constants $C_{1,3}, C_{1,5}, C_{2,1}, C_{2,3}, C_{2,5}, C_3$ of inconsequential value.  Note that
the constant $5/24$ in front of the nonlinearity in equation \eqref{eq:NLS} was chosen so as to cancel the `resonant' term $k=1$
in $E_n^1$.

Using \eqref{reg for vn} and making the necessary change of variables shows
\begin{equation}\label{two small errors}\begin{aligned}
\| E_n^2 \|_{L^1_t L^2_x (|t| \leq \frac{T}{3\xi_n})} &\lesssim \tfrac{\xi_n^{1/4}}{\xi_n} \| v_n \|_{L^5_t L^{10}_x ([-T,T])}^5 \lesssim_\phi \xi_n^{-3/4} \\
\| E_n^3 \|_{L^1_t L^2_x (|t| \leq \frac{T}{3\xi_n})} &\lesssim \xi_n^{3/4} \tfrac{T}{\xi_n} \| v_n \|_{L^\infty_t L^2_x ([-T,T])}\lesssim_\phi \xi_n^{-1/4}T.
\end{aligned}\end{equation}
Unlike these two terms, $E_n^1$ does not converge to zero in this norm.  Indeed, the simple arguments above show merely $\|E_n^1\|_{L^1_t L^2_x} \lesssim 1$.
Following \cite{Tao:2remarks}, the expedient way to deal with this error term is to alter our approximate solution $\tilde u_n$ on this middle interval;
ultimately we will see that the modification is negligible in all the important norms.

\begin{lemma}[Error-Control, \cite{Tao:2remarks}]\label{L:error-control}
Let $E_n$ be as defined above and let $e_n$ be the solution to the forced Airy equation
\begin{equation*}
\begin{cases}
(\partial_t+\partial_{xxx})e_n =E_n\\
e_n(0)=0.
\end{cases}
\end{equation*}
Then
\begin{equation*}
\lim_{n\to\infty} \Bigl[\|e_n\|_{L_t^\infty L_x^2(|t|\leq \frac{T}{3\xi_n})} + \bigl\||\partial_x|^{1/6}e_n\bigr\|_{L^6_{t,x}(|t|\leq \frac{T}{3\xi_n})}
    + \|e_n\|_{L^5_xL^{10}_t(|t|\leq \frac{T}{3\xi_n})}\Bigr]=0.
\end{equation*}
\end{lemma}

The proof of Lemma~\ref{L:error-control} uses the compactness of the time interval in an essential way.  Indeed, we already see the importance
of this in \eqref{two small errors}.  As noted earlier, it is unavoidable since the norms in which $v_n$ must be estimated are not scale-invariant.

By the Strichartz inequality (Lemma~\ref{L:linear estimates}), the bounds given in \eqref{two small errors} suffice to control
the contributions from $E^2_n$ and $E^3_n$.  Using linearity, one may therefore focus one's attention on $E_n^1$. To handle this
term one uses instead the oscillatory behaviour of the terms $e^{i3\xi_nx+i3\xi_n^3t}$ and $e^{i5\xi_nx+i5\xi_n^3t}$.  Indeed,
the frequencies $(\omega, \xi)=(3\xi_n^3, 3\xi_n)$ and $(\omega, \xi)=(5\xi_n^3, 5\xi_n)$ are far from the cubic $\omega=\xi^3$;
this fact together with $X^{s,b}$-type arguments are used to yield the claim in this case.  For details, see \cite[Lemma~6.1]{CCT}
or \cite[Lemma~3.1]{Tao:2remarks}.

\begin{lemma}[Spacetime bounds for $\tilde u_n$]\label{L:STB for tilde u_n}
For $\tilde u_n$ as above,
\begin{align*}
\bigl\| |\partial_x|^{1/6} \tilde u_n \bigr\|_{L_{t,x}^6(|t|\leq \frac{T}{3\xi_n})}
    + \bigl\|  \tilde u_n \bigr\|_{L_x^5 L_t^{10}(|t|\leq \frac{T}{3\xi_n})} \lesssim_\phi 1.
\end{align*}
\end{lemma}

\begin{proof}
As $ (\partial_t+\partial_{xxx})(\tilde u_n-e_n)=\mu \partial_x(\tilde u_n^5)$, the Strichartz inequality, Lemma~\ref{L:error-control}, and
\eqref{reg for vn} yield
\begin{align*}
\bigl\|   & |\partial_x|^{1/6} \tilde u_n \bigr\|_{L_{t,x}^6(|t|\leq \frac{T}{3\xi_n})}
   + \bigl\|  \tilde u_n \bigr\|_{L_x^5 L_t^{10}(|t|\leq \frac{T}{3\xi_n})}\\
&\lesssim \bigl\| |\partial_x|^{1/6} e_n \bigr\|_{L_{t,x}^6(|t|\leq \frac{T}{3\xi_n})}+ \|e_n\|_{L_x^5 L_t^{10}(|t|\leq \frac{T}{3\xi_n})}
    +\|\tilde u_n(0)\|_{L_x^2} + \bigl\|\partial_x (\tilde u_n^5)\bigr\|_{L_t^1 L_x^2(|t|\leq \frac{T}{3\xi_n})} \\
&\lesssim_\phi 1 + (\xi_n+\xi_n^{1/4})\|v_n (3\xi_nt,x)\|_{L_t^5 L_x^{10}(|t|\leq \frac{T}{3\xi_n})}^5\\
&\lesssim_\phi 1 + \tfrac{\xi_n+\xi_n^{1/4}}{\xi_n}\|v_n\|_{L_t^5 L_x^{10}(\R\times\R)}^5\\
&\lesssim_\phi 1.
\end{align*}
Note that changing variables in the time integral is responsible for the appearance of $\xi_n$ in the denominator on the penultimate line.
\end{proof}

This allows us to prove that $\tilde u_n -e_n$ is an approximate solution to gKdV on the middle time interval.

\begin{lemma}[Good approximation to gKdV -- the middle interval]\label{L:good approx middle}
Let $\tilde u_n$ and $e_n$ be as defined above.  Then $\tilde u_n -e_n$ approximately solves the gKdV equation \eqref{eq:gKdV}
in the sense that
\begin{align*}
\lim_{n\to \infty}\bigl\||\partial_x|^{-1}\bigl\{(\partial_t+ \partial_{xxx})(\tilde u_n-e_n)- \mu \partial_x \bigl[(\tilde u_n-e_n)^5\bigr]\bigr\}\bigr\|_{L_x^1L_t^2(|t|\leq \frac{T}{3\xi_n})}=0.
\end{align*}
\end{lemma}

\begin{proof}
A quick computation yields
$$
(\partial_t+ \partial_{xxx})(\tilde u_n-e_n)= \mu\partial_x \bigl[(\tilde u_n-e_n)^5\bigr] + \mu \partial_x \bigl[\tilde u_n^5 -(\tilde u_n-e_n)^5\bigr].
$$
Thus, in order to establish the claim we need to show that
\begin{align*}
\lim_{n\to \infty}\bigl\| \tilde u_n^5 - (\tilde u_n-e_n)^5 \bigr\|_{L_x^1L_t^2(|t|\leq \frac{T}{3\xi_n})}=0.
\end{align*}
This follows easily from Lemmas~\ref{L:error-control} and \ref{L:STB for tilde u_n} and H\"older's inequality.
\end{proof}

The next step toward invoking the perturbation theory is checking the proximity of the initial data.

\begin{lemma}[Agreement with the initial data]\label{L:initial match}
For $\tilde u_n$ as defined above,
\begin{align}\label{E:initial match}
\lim_{n\to \infty} \bigl\| u_n(t_n)-\tilde u_n(t_n)\bigr\|_{L_x^2}=0.
\end{align}
Recall that $u_n(t_n)$ is defined in \eqref{E:initial data}.
\end{lemma}

\begin{proof}
We break the proof in two cases depending on whether or not $T_0$ is finite.

Consider first the case $|T_0|<\infty$.  Note that in this case we must necessarily have $t_n\to 0$ as $n\to \infty$.
Requiring $T>|T_0|$ and $n$ sufficiently large, and using the definition of $\tilde u_n$ and Plancherel, we estimate
\begin{align*}
\bigl\| u_n(t_n)-\tilde u_n(t_n)\bigr\|_{L_x^2}
&\leq \bigl\| e^{-t_n\partial_x^3}(e^{ix\xi_n}\phi(x)) - e^{ix\xi_n+it_n\xi_n^3} v_n(3\xi_nt_n, x+3\xi_n^2t_n) \bigr\|_{L_x^2}\\
&= \bigl\| e^{it_n(\xi+\xi_n)^3}\hat \phi(\xi) - e^{it_n\xi_n^3+ 3it_n\xi_n^2\xi } \hat{v}_n(3\xi_nt_n, \xi) \bigr\|_{L_\xi^2}\\
&= \bigl\| e^{it_n\xi^3} e^{3it_n\xi_n\xi^2}\hat \phi(\xi) - \hat{v}_n(3\xi_nt_n, \xi) \bigr\|_{L_\xi^2}.
\end{align*}
Now recall that $3\xi_nt_n\to T_0$ and, by construction, $v(T_0) = e^{-iT_0\partial_x^2}\phi$; these combined with \eqref{v_n to v}
yield the claim the $|T_0|<\infty$ case.

Next, we consider the case $T_0=\infty$; the case $T_0=-\infty$ can be handled identically.  Using the unitarity of $e^{-t\partial_x^3}$
and the calculation above, we obtain
\begin{align*}
\bigl\| u_n(t_n)-\tilde u_n(t_n)\bigr\|_2
&\leq \bigl\| e^{-\frac{T}{3\xi_n}\partial_x^3}(e^{ix\xi_n} \phi) - \tilde u_n\bigl(\tfrac{T}{3\xi_n} \bigr)  \bigr\|_2
=\bigl\| e^{iT\xi^2 + i\frac{T}{3\xi_n} \xi^3}\hat{\phi}(\xi) - \hat{v}_n(T,\xi)\bigr\|_2.
\end{align*}
Recalling the construction of $v$ in this case together with the fact that $\xi_n\to \infty$ by hypothesis,
the dominated convergence theorem combined with \eqref{v_n to v} yield \eqref{E:initial match}.
\end{proof}

We are now in a position to apply the stability result Theorem~\ref{T:stab}.  We begin with the case $|T_0|<\infty$, which implies that
$t_n$ lies in the interval $|t|\leq \frac{T}{3\xi_n}$ for $T$ and $n$ large enough.  In this case, we use $\tilde u_n-e_n$ as our approximate
solution on the time interval $|t|\leq \frac{T}{3\xi_n}$.  By Lemma~\ref{L:good approx middle}, for $n$ sufficiently large
this is an approximate solution to gKdV, while by Lemmas~\ref{L:error-control} and \ref{L:initial match}, we have asymptotic (in $n$)
agreement of the initial data.  Thus we obtain a solution $u_n$ to gKdV on the interval $|t|\leq \frac{T}{3\xi_n}$ which matches the initial
data stated in the theorem and obeys
$$
\lim_{n\to \infty}\Bigl(\|u_n-\tilde u_n\|_{L_t^\infty L_x^2(|t|\leq \frac{T}{3\xi_n})} + \|u_n-\tilde u_n\|_{L_x^5 L_t^{10}(|t|\leq \frac{T}{3\xi_n})}
+\bigl\||\partial_x|^{1/6}(u_n-\tilde u_n)\|_{L_{t,x}^6(|t|\leq \frac{T}{3\xi_n})}\Bigr)=0.
$$
Note that we used Lemma~\ref{L:error-control} to remove $e_n$ from the formula above.  To extend the solution $u_n$ to the whole real line,
we use the formula above together with Lemma~\ref{L:good approx tail} and Theorem~\ref{T:stab}; moreover,
\begin{align}\label{E:un agrees}
\lim_{n\to \infty}\Bigl(\|u_n-\tilde u_n\|_{L_t^\infty L_x^2(\R\times\R)} + \|u_n-\tilde u_n\|_{L_x^5 L_t^{10}(\R\times\R)}
+\bigl\||\partial_x|^{1/6}(u_n-\tilde u_n)\|_{L_{t,x}^6(\R\times\R)}\Bigr)=0.
\end{align}

The argument in the case $T_0=\pm\infty$ is very similar.  One simply treats the three time intervals in a different order.  We still obtain a global
solution $u_n$ to gKdV with satisfies \eqref{E:un agrees}.

We are left with the task of constructing the compactly supported approximation to our solution.  The asymmetry in the space/time
exponents in $L_x^5L_t^{10}$ combined with the boost in \eqref{E:tilde u defn} prevent us from using a simple density argument.

Given $\eps>0$, let $T>0$ and $n$ be sufficiently large so that
$$
\|u_n\|_{L_x^5L_t^{10}(|t|>\frac{T}{3\xi_n})}\leq \eps.
$$
This is possible by virtue of  \eqref{E:un agrees} and the proof of Lemma~\ref{L:good approx tail}.  This allows us to discount the region
$|t|>\frac{T}{3\xi_n}$ from further consideration.  In light of the $L_{t,x}^6$ bounds on $v$, we may choose $\psi_\eps\in C_c^\infty((-T,T)\times\R)$
so that
$$
\|v -\psi_\eps\|_{L_{t,x}^6([-T,T]\times\R)}\leq \eps.
$$
In particular, by \eqref{v_n to v}, for $n$ sufficiently large depending on $\eps$,
\begin{align}\label{v_n-psi}
\|v_n -\psi_\eps\|_{L_{t,x}^6([-T,T]\times\R)}\leq 2\eps.
\end{align}

By the triangle inequality,
\begin{align*}
\bigl\| u_n(t,x) &- \Re \bigl[e^{ix\xi_n+it\xi_n^3} \psi_\eps(3\xi_n t, x+3\xi_n^2 t)\bigr]\bigr\|_{L_x^5 L_t^{10}(|t|\leq \frac{T}{3\xi_n})} \\
&\leq  \bigl\| u_n(t,x) - \tilde u_n(t,x) \bigr\|_{L_x^5 L_t^{10}(|t|\leq \frac{T}{3\xi_n})} \\
&\qquad + \bigl\| v_n(3\xi_n t, x+ 3\xi_n^2 t) - \psi_\eps(3\xi_n t, x+ 3\xi_n^2 t) \bigr\|_{L_x^5 L_t^{10}(|t|\leq \frac{T}{3\xi_n})}.
\end{align*}
The former difference converges to zero by \eqref{E:un agrees}; the latter we estimate using \eqref{v_n-psi} as follows:
\begin{align*}
\bigl\| & v_n(3\xi_n t,x+ 3\xi_n^2 t) - \psi_\eps(3\xi_n t, x+ 3\xi_n^2 t) \bigr\|_{L_x^5 L_t^{10}(|t|\leq \frac{T}{3\xi_n})} \\
&\lesssim \xi_n^{-1/10} \bigl\| v_n(t, x+ \xi_n t) - \psi_\eps(t, x+ \xi_n t) \bigr\|_{L_x^5 L_t^{10}(|t|\leq T)} \\
&\lesssim \xi_n^{-1/10} \bigl\| v_n- \psi_\eps \bigr\|_{L_{t,x}^6(|t|\leq T)}^{3/5}
    \Bigl\{ \bigl\| v_n(t, x+ \xi_n t) \bigr\|_{L_x^4 L_t^{\infty}(|t|\leq T)}^{2/5}
        + \big\| \psi_\eps(t, x+ \xi_n t) \bigr\|_{L_x^4 L_t^\infty(|t|\leq T)}^{2/5} \Bigr\} \\
& \lesssim_\psi \xi_n^{-1/10} \eps^{3/5}  \bigl\{ (\xi_n^{1/4}+\xi_n)^{1/10} + \xi_n^{1/10} \bigr\};
\end{align*}
to obtain the last inequality, we used the fact that $\|\psi_\eps(t, x+ \xi_n t)\|_{L^\infty_t}$ has support of diameter $O(\xi_n)$
and Lemma~\ref{L:Gulkan}.  When using Lemma~\ref{L:Gulkan}, the boost is accounted for by using the Galilei symmetry of the Schr\"odinger equation
and \eqref{reg for vn}.

This completes the proof of Theorem~\ref{T:NLS emb} and with it, the section.
\end{proof}

%
%
%
%

\section{Reduction to almost periodic solutions}\label{S:reduct}

In this section we prove Theorem~\ref{T:reduct}, which we will derive as a consequence to the following key proposition,
asserting a certain compactness (modulo symmetries) in sequences of almost blowup solutions with mass converging to the critical mass from below.

\begin{proposition}[Palais--Smale condition modulo symmetries]\label{P:Palais Smale}
Fix $\mu=\pm 1$ and assume Conjecture~\ref{Conj:gKdV} fails for this choice of $\mu$.  Let $M_c$ denote the corresponding critical mass
and assume that Conjecture~\ref{Conj:NLS} holds for initial data with mass $M(v_0)\leq 2M_c$.  Let $u_n: \R\times \R\to \R$ be a
sequence of solutions to \eqref{eq:gKdV} and $t_n$ a sequence of times such that $M(u_n)\nearrow M_c$ and
\begin{equation*}
\lim_{n\to\infty} S_{\ge t_n}(u_n)=\lim_{n\to\infty} S_{\le t_n} (u_n)=\infty.
\end{equation*}
Then the sequence $u_n(t_n)$ has a subsequence which converges in $L^2_x$ modulo the symmetries described in Definition~\ref{D:sym}.
\end{proposition}

\begin{proof}
Using the time-translation symmetry of \eqref{eq:gKdV}, we may set $t_n=0$ for all $n\geq 1$.  Thus,
\begin{equation}\label{blow up in two}
\lim_{n\to \infty} S_{\ge 0} (u_n) = \lim_{n\to \infty} S_{\le 0}(u_n) = \infty.
\end{equation}

Applying Lemma~\ref{L:profile decomp} to the sequence $u_n(0)$ (which is bounded in $L^2_x(\R)$) and passing to a
subsequence if necessary, we obtain the decomposition
\begin{align}\label{decomp}
u_n(0) = \sum_{1\le j\le J} g_n^je^{-t_n^j\partial_x^3}[\Re(e^{ix\xi_n^j\lambda_n^j}\phi^{j})]+ w_n^J.
\end{align}
By \eqref{orthog 2} in Lemma~\ref{L:profile decomp}, we have mass decoupling:
\begin{equation}\label{mass decoupling}
\limsup_{n \to \infty}\sum_{j=1}^{\infty} M[\Re(e^{ix\xi_n^j\lambda_n^j}\phi^j)]\leq \limsup_{n \to \infty} M( u_n(0) )\leq M_c.
\end{equation}
In particular, we have $\sup_{j\geq 1} \limsup_{n \to \infty} M[\Re(e^{ix\xi_n^j\lambda_n^j}\phi^j)] \le M_c$.  The argument now
breaks into two cases depending on whether equality occurs here.

\medskip

{\bf Case I.}  Assume
$$
\sup_{j\geq 1} \limsup_{n \to \infty} M[\Re(e^{ix\xi_n^j\lambda_n^j}\phi^j)] =M_c.
$$
Comparing this with \eqref{mass decoupling}, we see that we must have $\phi^j=0$ for $j\geq 2$, that is, there is only one linear profile
and it carries all the mass.  Consequently, the linear profile decomposition simplifies to
\begin{equation}\label{eq:loc-3}
u_n(0)=g_n e^{-t_n\partial_x^3}[\Re(e^{ix\xi_n\lambda_n}\phi)]+w_n \quad\text{with}\quad \|w_n\|_{L^2_x}\to 0.
\end{equation}
By applying the symmetry operation $T_{g_n^{-1}}$ to $u_n$, which does not affect the hypotheses of Proposition~\ref{P:Palais Smale},
we may take all $g_n$ to be the identity.  Thus, \eqref{eq:loc-3} reduces to
\begin{equation}\label{eq:loc-3'}
u_n(0)= e^{-t_n\partial_x^3}[\Re(e^{ix\xi_n}\phi)] + w_n,
\end{equation}
for some sequence $\{t_n\}_{n\geq 1} \subset \R$, some $\{\xi_n\}_{n\geq 1}\subset [0,\infty)$ such that either $\xi_n\equiv 0$
or $\xi_n\to \infty$, and some $\phi, w_n\in L_x^2$ with $M(w_n) \to 0$ (and hence $S_\R(e^{-t\partial_x^3} w_n) \to 0$) as $n \to \infty$.

\smallskip

{\bf Case I a).} We first consider the case when $\xi_n\to \infty$ as $n\to \infty$.  By passing to a subsequence if necessary,
we may assume that the sequence $\{3t_n\xi_n\}_{n\geq 1}$ converges to some $T_0\in [-\infty, +\infty]$.

A computation reveals that
\begin{align*}
M\bigl[\Re (e^{ix\xi_n}\phi) \bigr]= \tfrac12 M(\phi) + \tfrac 12 \int_\R \Re(e^{2ix\xi_n}\phi(x)^2)\, dx.
\end{align*}
Thus, invoking the Riemann--Lebesgue lemma together with the fact that, by assumption, $\lim_{n\to \infty} M\bigl[\Re (e^{ix\xi_n}\phi) \bigr] = M_c$,
we derive
\begin{align*}
M(\phi)=2M_c.
\end{align*}
This places us in the setting of Theorem~\ref{T:NLS emb}.  Combining this with the fact that $M(w_n)\to 0$ as $n\to \infty$ and the
stability result Theorem~\ref{T:stab}, gives
$$
S_{\R}(u_n)\lesssim_{M_c}1,
$$
thus contradicting \eqref{blow up in two}.

\smallskip

{\bf Case I b).} We are left to consider the case when $\xi_n\equiv 0$.  Thus, \eqref{eq:loc-3'} reduces to
\begin{equation}\label{eq:loc-3''}
u_n(0)= e^{-t_n\partial_x^3}[\Re(\phi)] + w_n,
\end{equation}
for some sequence $\{t_n\}_{n\geq 1} \subset \R$ and some $\phi, w_n\in L_x^2$ with $M(w_n) \to 0$.  By passing to a subsequence if necessary,
we may assume that the sequence $\{t_n\}_{n\geq 1}$ converges to some $T_0\in [-\infty, +\infty]$.  If $T_0\in (-\infty, \infty)$,
then by replacing $\phi$ by $e^{-T_0 \partial_x^3}\phi$, we may assume that $T_0=0$; moreover, absorbing the error
$e^{-t_n \partial_x^3}[\Re(\phi)] - \Re(\phi)$ into the error term $w_n$, we may reduce to $t_n\equiv 0$.  To review, we may assume that either
$t_n\equiv 0$ or $t_n\to\pm \infty$.  We treat these two scenarios separately.

{\bf Case I b1).} Assume $t_n\equiv 0$.  Then \eqref{eq:loc-3''} becomes
$$
u_n(0)= \Re(\phi)+w_n
$$
with $M[\Re(\phi)]=M_c$ and $M(w_n) \to 0$ as $n \to \infty$.  This immediately implies that $u_n(0)$ converges to $\Re(\phi)$ in $L_x^2$,
thus yielding the claim of Proposition~\ref{P:Palais Smale}.

{\bf Case I b2).} Assume $t_n\to \pm \infty$ as $n\to \infty$.  We only present the argument for $t_n\to \infty$;
the case $t_n\to -\infty$ can be treated symmetrically.

By the Strichartz inequality,
$$
S_{\R}\bigl(e^{-t\partial_x^3}[\Re(\phi)]\bigr)<\infty,
$$
and hence, by the monotone convergence theorem,
$$
\lim_{n\to\infty}S_{\ge 0}\bigl(e^{-t\partial_x^3}e^{-t_n\partial_x^3}[\Re(\phi)]\bigr)=0.
$$
Invoking $S_\R(w_n)\to 0$ and \eqref{eq:loc-3''}, we obtain
$$
\lim_{n\to\infty}S_{\ge 0}\bigl(e^{-t\partial_x^3}u_n(0)\bigr)=0.
$$
Applying the stability result Theorem~\ref{T:stab} (using $0$ as the approximate solution and $u_n(0)$ as the initial data), we conclude
$$
\lim_{n\to\infty} S_{\ge 0}(u_n)=0,
$$
which contradicts \eqref{blow up in two}.

\medskip

{\bf Case II.} We  now turn to the case where $u_n$ contains multiple profiles, namely, when
\begin{equation}\label{Palais:Mj}
\sup_{j\geq 1} \limsup_{n \to \infty}  M[\Re(e^{ix\xi_n^j\lambda_n^j}\phi^j)] \leq M_c - \eps \quad \text{for some } \eps>0.
\end{equation}
We will eventually show that this leads to a contradiction.

Reordering the indices in the decomposition \eqref{decomp} if necessary, we may assume that there exists $1\leq J_0\leq J$
such that for each $1\leq j\leq J_0$ we have $\xi_n^j\equiv 0$, while for $J_0<j\leq J$ we have $\xi_n^j\lambda_n^j\to \infty$ as $n\to \infty$.
Note that both the reordering and $J_0$ depend upon $J$.

\medskip

For $1\leq j\leq J_0$ we make the following reductions: First, refining the subsequence once for each $j$ and using a diagonal argument,
we may assume that for each $j$, the sequence $\{\tnj\}_{n\geq 1}$ converges to some $T^j\in [-\infty, \infty]$.  If $T^j\in (-\infty, \infty)$,
then by replacing $\phi^j$ by $e^{-T_j \partial_x^3}\phi^j$, we may assume that $T^j=0$; moreover, absorbing the difference
$e^{-\tnj \partial_x^3}[\Re(\phi^j)] -\Re(\phi^j)$ into the error term $\wnJ$, we may assume that $\tnj\equiv 0$.
Thus, either $\tnj\equiv 0$ or $\tnj\to\pm \infty$.

Continuing with the case $1\leq j\leq J_0$, we define the nonlinear profiles $v^j$ as follows:

\begin{CI}
\item If $\tnj\equiv 0$, then $v^j$ is the maximal-lifespan solution to \eqref{eq:gKdV} with initial data $v^j(0)=\Re(\phi^j)$.
\item If $\tnj\to \infty$, then $v^j$ is the maximal-lifespan solution to \eqref{eq:gKdV} that scatters forward in time to $e^{-t\partial_x^3}\Re(\phi^j)$.
\item If $\tnj\to -\infty$, then $v^j$ is the maximal-lifespan solution to \eqref{eq:gKdV} that scatters backward in time to $e^{-t\partial_x^3}\Re(\phi^j)$.
\end{CI}

By  \eqref{Palais:Mj}, each $v^j$ has mass less than $M_c$; as a consequence it is global and $S_\R(v^j)<\infty$.  Moreover,
combining this with the small data theory (see Theorem~\ref{T:local}) gives
$$
S_\R(v^j)\lesssim M(v^j)^{5/2}=M[\Re (\phi^j)]^{5/2}\lesssim_{M_c} M[\Re (\phi^j)].
$$
Next, for each $1\leq j\leq J_0$ and $n\geq 1$, we introduce $\vnj:\R\times\R\to \R$ defined by
$$
\vnj(t):= T_{\gnj}\bigl[ v^j(\cdot + \tnj)\bigr](t).
$$
Each $\vnj$ is a global solution to \eqref{eq:gKdV} with initial data $\vnj(0)=\gnj v^j(\tnj)$.  Furthermore,
\begin{equation}\label{vnj bounds -1}
S_{\R}(v_n^j)=S_\R(v^j)\lesssim_{M_c} M[\Re (\phi^j)].
\end{equation}

\medskip

Now consider $J_0<j\leq J$.  In this case we make the following reduction: refining the subsequence once for every $j$ and using a diagonal
argument, we may assume that for each $j$, the sequence $\{3\tnj\xi_n^j\lambda_n^j\}_{n\geq 1}$ converges to some $T^j\in [-\infty, \infty]$.

A computation reveals that
\begin{align*}
M\bigl[\Re (e^{ix\xi_n^j\lambda_n^j}\phi^j) \bigr]= \tfrac12 M(\phi^j) + \tfrac 12 \int_\R \Re(e^{2ix\xi_n^j\lambda_n^j}\phi^j(x)^2)\, dx.
\end{align*}
Thus, by \eqref{Palais:Mj} and the Riemann--Lebesgue lemma, $M(\phi^j)<2M_c$.  This places us in the setting of Theorem~\ref{T:NLS emb}.
Hence, for $n$ sufficiently large there exists a global solution $\tilde v_n^j$ to gKdV with initial data
$$
\tilde v_n^j(t_n^j) = e^{-t_n^j\partial_x^3} [\Re(e^{ix\xi_n^j\lambda_n^j}\phi^j)].
$$
Moreover, these solutions obey global spacetime bounds.  Combining this with the small data theory guarantees that, for $n$ sufficiently large,
\begin{align}\label{tilde v bound}
S_{\R}(\tilde v_n^j)\lesssim_{M_c} M[\Re(e^{ix\xi_n^j\lambda_n^j}\phi^j)].
\end{align}
Next, we define the nonlinear profiles $v_n^j:\R\times\R \to \R$ by
$$
\vnj(t):= T_{\gnj}\bigl[ \tilde v_n^j(\cdot + \tnj)\bigr](t).
$$
Invariance of the scattering norm under symmetries shows that \eqref{tilde v bound} can be recast as
\begin{equation}\label{vnj bounds -2}
S_\R(v^j_n)\lesssim_{M_c} M[\Re(e^{ix\xi_n^j\lambda_n^j}\phi^j)]
\end{equation}
for $n$ sufficiently large.

By Lemma~\ref{L:decoupling}, we have decoupling of the nonlinear profiles defined above.  More precisely, due to the
orthogonality conditions in Lemma~\ref{L:profile decomp},
\begin{align}\label{decoupling}
\lim_{n\to \infty} \bigl\| v_n^j v_n^k \bigr\|_{L_x^{5/2} L_t^5(\R\times\R)}=0 , \quad \text{for all} \quad 1\leq j\neq k\leq J.
\end{align}
Note that the three cases discussed in Lemma~\ref{L:decoupling} cover the possible relations between $j$, $k$, and $J_0$.
This decoupling property will allow us to show that $u_n$ may be well approximated by a sum of the $v_n^j$.  To this end,
we define an approximate solution
\begin{equation}\label{approximate-solution-gKdV}
u_n^J(t):=\sum_{1\le j\le J}v_n^j(t)+e^{-t\partial_x^3}w_n^J.
\end{equation}
Next we will show that $u_n^J$ is indeed a good approximation to $u_n$ for $n,J$ sufficiently large.

\begin{lemma}[Asymptotic agreement with initial data]\label{L:M approx}
For any $J\geq 1$ we have
$$ \lim_{n \to \infty} M\bigl( u^J_n(0) - u_n(0) \bigr) = 0.$$
\end{lemma}

\begin{proof}
This follows directly from
\begin{equation}\label{E:match at 0}
u_n(0) - \sum_{j=1}^J v_n^j(0) - w_n^J \longrightarrow 0 \quad\text{in $L^2_x$ as $n\to\infty$,}
\end{equation}
which is a consequence of the way $v_n^j$ were constructed.
\end{proof}

Next we show that $u_n^J$ has finite scattering size for $n,J$ sufficiently large.  Indeed, by Remark~\ref{R:wnj small} combined with
\eqref{decoupling}, \eqref{vnj bounds -1}, \eqref{vnj bounds -2}, and \eqref{orthog 2},
\begin{align}\label{unj finite STB}
\lim_{J\to\infty}\limsup_{n\to\infty} S_\R(u_n^J)
&\lesssim \lim_{J\to\infty}\limsup_{n\to\infty} \ \biggl\{S_\R\Bigl(\sum_{1\le j\le J}v_n^j\Bigr) + S_\R\bigl( e^{-t\partial_x^3}w_n^J\bigr)\biggr\}\nonumber \\
&\lesssim \lim_{J\to\infty}\limsup_{n\to\infty}\sum_{1\le j\le J}S(v_n^j)\nonumber \\
&\lesssim_{M_c} \lim_{J\to\infty} \limsup_{n\to\infty} \sum_{1\le j\le J}M[\Re(e^{ix\xi_n^j\lambda_n^j}\phi^j)]
\lesssim_{M_c}1.
\end{align}

The last step before invoking the stability result Theorem~\ref{T:stab} is to check that $u_n^J$ almost solves the equation.

\begin{lemma}[Asymptotic solution to the equation]\label{L:Eq approx}
We have
$$
\lim_{J \to \infty} \limsup_{n \to \infty}
\bigl\| |\partial_x|^{-1} \bigl[(\partial_t + \partial_{xxx}) u^J_n - \partial_x\bigl((u^J_n)^5\bigr)\bigr] \bigr\|_{L_x^1L^2_t(\R \times \R)} = 0.
$$
\end{lemma}

\begin{proof}
For $J,n\geq 1$,
$$
(\partial_t+\partial_{xxx})u_n^J=\sum_{1\le j\le J}\partial_x\bigl( (v_n^j)^5 \bigr).
$$
Thus it suffices to show that
\begin{equation*}
\lim_{J\to\infty}\limsup_{n\to\infty} \Bigl\|(u_n^J)^5-\sum_{1\le j\le J}(v_n^j)^5\Bigr\|_{L_x^1L^2_t(\R \times \R)}=0,
\end{equation*}
which, by the triangle inequality, reduces to proving
\begin{equation}\label{eq:loc-2}
\lim_{J\to\infty}\limsup_{n\to\infty} \Bigl\|\bigl(u_n^J-e^{-t\partial_x^3}w_n^J\bigr)^5-(u_n^J)^5\Bigr\|_{L_x^1L^2_t(\R \times \R)}=0
\end{equation}
and
\begin{equation}\label{eq:loc-1}
\lim_{J\to\infty}\limsup_{n\to\infty} \Bigl\|\Bigl(\sum_{1\le j\le J} v_n^j\Bigr)^5 - \sum_{1\le j\le J}(v_n^j)^5\Bigr\|_{L_x^1L^2_t(\R \times \R)}=0.
\end{equation}

We first consider \eqref{eq:loc-2}.  This follows from the pointwise inequality
$$
\Bigl|\bigl(u_n^J-e^{-t\partial_x^3}w_n^J\bigr)^5-(u_n^J)^5\Bigr|
\lesssim \bigl|e^{-t\partial_x^3}w_n^J\bigr|^5 + \bigl|e^{-t\partial_x^3}w_n^J\bigr| \bigl|u_n^J\bigr|^4
$$
together with H\"older's inequality, Remark~\ref{R:wnj small}, and \eqref{unj finite STB}.

We now turn to \eqref{eq:loc-1}.  We observe the following pointwise inequality:
\begin{align*}
\Bigl| \Bigl(\sum_{1\le j\le J} v_n^j\Bigr)^5 - \sum_{1\le j\le J}(v_n^j)^5\Bigr|
\leq \sum_{i_1,i_2,i_3=1}^J\ \sum_{1\leq j\neq k\leq J} |v_n^{i_1}| |v_n^{i_2}||v_n^{i_3}| |\vnj v_n^k|.
\end{align*}
By H\"older's inequality combined with \eqref{vnj bounds -1}, \eqref{vnj bounds -2}, and \eqref{decoupling},
we see that this vanishes asymptotically as $n\to \infty$ in $L_x^1L_t^2$.

This proves \eqref{eq:loc-1} and completes the proof of the lemma.
\end{proof}

We are now in a position to apply the stability result Theorem~\ref{T:stab}.  Indeed, using \eqref{unj finite STB}
together with Lemmas~\ref{L:M approx} and \ref{L:Eq approx}, we deduce that for $J$ and $n$ sufficiently large, $u_n^J$ is an approximate
solution to \eqref{eq:gKdV} satisfying the hypotheses of Theorem~\ref{T:stab}.  Thus, for $n$ sufficiently large, we obtain
$$
S_\R(u_n)\lesssim_{M_c}1,
$$
which contradicts \eqref{blow up in two}.  Thus, Case II cannot occur and we have finished the proof of Proposition~\ref{P:Palais Smale}.
\end{proof}

With the Palais--Smale condition modulo symmetries in place, we are now ready to prove Theorem~\ref{T:reduct}.

\begin{proof}[Proof of Theorem~\ref{T:reduct}]
As discussed in the introduction, failure of Conjecture~\ref{Conj:gKdV} implies the existence of a critical mass $M_c$ and
a sequence $u_n: \R\times \R \to \R$ of solutions with $M(u_n) \nearrow M_c$ and $\lim_{n \to \infty} S_\R(u_n) = +\infty$.
Choose $t_n\in \R$ so that $S_{\geq t_n}(u_n)= S_{\leq t_n}(u_n)$.  Then,
\begin{align}\label{hyp 2}
\lim_{n\to \infty} S_{\ge t_n}(u_n)=\lim_{n\to \infty}S_{\le t_n}(u_n)=\infty.
\end{align}
Using the time-translation symmetry of \eqref{eq:gKdV}, we may take all $t_n=0$. Applying Proposition~\ref{P:Palais Smale}, and
passing to a subsequence if necessary, we can locate $u_0 \in L^2_x(\R)$ such that $u_n(0)$ converge in $L^2_x(\R)$ modulo
symmetries to $u_0$; thus, there exist group elements $g_n \in G$ such that $g_n u_n(0)$ converge strongly in $L^2_x(\R)$ to
$u_0$.  Applying the symmetry operation $T_{g_n^{-1}}$ to the solution $u_n$ we may take all $g_n$ to be the identity, and thus
$u_n(0)$ converge strongly in $L^2_x(\R)$ to $u_0$. In particular this implies $M(u_0) \leq M_c$.

Let $u: I \times \R \to \R$ be the maximal-lifespan solution with initial data $u(0) = u_0$ as given by Theorem~\ref{T:local}.
We claim that $u$ blows up both forward and backward in time.  Indeed, if $u$ does not blow up forward in time (say), then by
Theorem~\ref{T:local} we have $[0,+\infty) \subset I$ and $S_{\geq 0}(u) < \infty$.  By Theorem~\ref{T:stab}, this implies for
sufficiently large $n$ that
$$ \limsup_{n \to \infty} S_{\geq 0}( u_n ) < \infty,$$
contradicting \eqref{hyp 2}.  Similarly if $u$ blows up backward in time.  By the definition of $M_c$ this forces $M(u_0) \geq
M_c$, and hence $M(u_0)$ must be exactly $M_c$.

It remains to show that our solution $u$ is almost periodic modulo symmetries.  Consider an arbitrary sequence of times $t'_n\in I$.
Now, since $u$ blows up both forward and backward in time, we have
$$ S_{\geq t'_n}(u) = S_{\leq t'_n}(u) = \infty.$$
Applying Proposition~\ref{P:Palais Smale} once again we see that $u(t'_n)$ has a subsequence which converges modulo symmetries.
Thus, the orbit $\{ u(t): t \in I \}$ is precompact in $L^2_x(\R)$ modulo symmetries.
\end{proof}

%
%
%
%

\section{Three enemies}\label{S:enemies}

In this section we outline the proof of Theorem~\ref{T:enemies}.  The argument closely follows \cite[\S4]{KTV}, which may be consulted
for further details.

Let $v:J\times\R\to\R$ denote a minimal-mass blowup solution whose existence (under the hypotheses of Theorem~\ref{T:enemies})
is guaranteed by Theorem~\ref{T:reduct}.  We denote the symmetry parameters of $v$ by $N_v(t)$ and $x_v(t)$. We will construct
our solution $u$ by taking a subsequential limit of various normalizations of $v$:

\begin{definition}
Given $t_0\in J$, we define the \emph{normalization} of $v$ at $t_0$ by
\begin{equation}\label{untn}
v^{[t_0]} := T_{g_{-x_v(t_0)N_v(t_0),N_v(t_0)}}\bigr( v( \cdot + t_0) \bigr).
\end{equation}
This solution is almost periodic modulo symmetries and has symmetry parameters
$$
N_{v^{[t_0]}}(t) = \frac{N_v(t_0+tN_v(t_0)^{-3})}{N_v(t_0)} \text{ and } x_{v^{[t_0]}}(t) =
N_v(t_0)[x_v(t_0+tN_v(t_0)^{-3})-x_v(t_0)].
$$
\end{definition}

Note that by the definition of almost periodicity, any sequence of $t_n\in J$ admits a subsequence so that $v^{[t_n]}(0)$
converges in $L^2_x$.  Furthermore, if $u_0$ denotes this limit and $u:I\times\R\to\R$ denotes the maximal-lifespan solution
with $u(0)=u_0$, then $u$ is almost periodic modulo symmetries with the same compactness modulus function as $v$.  Lastly, Theorem~\ref{T:stab}
shows that $v^{[t_n]}\to u$ in critical spacetime norms (along the subsequence) uniformly on any compact subset of $I$.

As in \cite[Corollary~3.6]{KTV}, $N_u(t)$ has the following local constancy property: there exists a small number $\delta$,
depending on $u$, such that for every $t_0 \in I$ we have
\begin{equation*}
\bigl[t_0 - \delta N(t_0)^{-3}, t_0 + \delta N(t_0)^{-3}\bigr] \subset I
\end{equation*}
and
\begin{equation*}
N(t) \sim_u N(t_0)
\end{equation*}
whenever $|t-t_0| \leq \delta N(t_0)^{-3}$.

Our first goal is to find a soliton-like solution from among the normalizations of $v$ if this is at all possible.  To this end, for any $T
\geq 0$, we define the quantity
\begin{equation}\label{cdef}
\osc(T) := \inf_{t_0 \in J} \,\frac{\sup\, \{ N_v(t) : t \in J \text{ and } |t-t_0| \leq T N_v(t_0)^{-3} \}}
    {\inf\, \{ N_v(t) : t \in J \text{ and } |t-t_0| \leq T N_v(t_0)^{-3} \}},
\end{equation}
which measures the least possible oscillation that one can find in $N_v(t)$ on time intervals of normalized duration $T$.

\medskip

{\bf Case 1:} $\lim_{T \to\infty} \osc(T) < \infty$.  Under this hypothesis, we will be able to extract a soliton-like solution.

Choose $t_n$ so that
$$
\limsup_{n\to\infty} \frac{\sup\, \{ N_v(t) : t \in J \text{ and } |t-t_n| \leq n N_v(t_n)^{-3} \}}
    {\inf\, \{ N_v(t) : t \in J \text{ and } |t-t_n| \leq n N_v(t_n)^{-3} \}} <\infty.
$$
Then a few computations reveal that any subsequential limit $u$ of $v^{[t_n]}$ fulfils the requirements to be classed as a
soliton-like solution in the sense of Theorem~\ref{T:enemies}.  In particular, $u$ is global because an almost periodic (modulo symmetries)
solution cannot blow up in finite time without its frequency scale function converging to infinity.

\medskip

When $\osc(T)$ is unbounded, we must seek a solution belonging to one of the remaining two scenarios. To aid in distinguishing
between them, we consider the quantity
$$
a(t_0) := \frac{ \inf_{t \in J: t \leq t_0} N_v(t) + \inf_{t \in J: t \geq t_0} N_v(t) }{N_v(t_0)}
$$
associated to each $t_0 \in J$.  This measures the extent to which $N_v(t)$ decays to zero on both sides of $t_0$. Clearly, this
quantity takes values in the interval $[0,2]$.

First we treat the case where $a(t_0)$ can be arbitrarily small.  As we will see, this will lead to a double cascade.

\medskip

{\bf Case 2:} $\lim_{T\to \infty} \osc(T) = \infty$ and $\inf_{t_0 \in J} a(t_0) = 0$.  From the behavior of $a(t_0)$ we may
choose sequences $t_n^{-}<t_n<t_n^{+}$ from $J$ so that $a(t_n)\to 0$, $N_v(t_n^{-})/N_v(t_n)\to 0$, and
$N_v(t_n^{+})/N_v(t_n)\to 0$.  Next we choose times $t_n'\in(t_n^{-},t_n^{+})$ so that
\begin{align}\label{N from below}
N_v(t_n') \geq \tfrac12 \sup\, \{ N_v(t) : t\in [t_n^{-},t_n^{+}] \}.
\end{align}
In particular, $N_v(t_n') \geq \frac 12N_v(t_n)$, which allows us to deduce that
\begin{align}\label{N to infty}
\frac{ N_v(t_n^{-}) }{ N_v(t_n') }\to 0 \quad\text{and}\quad \frac{ N_v(t_n^{+}) }{ N_v(t_n') } \to 0.
\end{align}

Now consider the normalizations $v^{[t_n']}$ and let $s_n^\pm := (t_n^\pm - t_n') N_v(t_n')^3$.  From \eqref{N from below} and
\eqref{N to infty} we see that
$$
N_{v^{[t_n']}}(s)\lesssim 1 \text{ for } s\in [s_n^-,s_n^+] \qquad \text{and} \qquad N_{v^{[t_n']}}(s_n^\pm)\to 0 \text{ as } n\to \infty.
$$
Passing to a subsequence if necessary, we obtain that $v^{[t'_n]}$ converge locally uniformly to a maximal-lifespan solution $u$
of mass $M(v)$ defined on an open interval $I$ containing $0$, which is almost periodic modulo symmetries.  Now $s_n^\pm$ must
converge to the endpoints of the interval $I$, which implies that $N_u(t)$ is bounded above on $I$ and thus, $u$ is global.
Rescaling $u$ slightly, we may ensure that $N_u(t)\leq 1$ for all $t\in\R$.

From the fact that $\osc(T)\to\infty$, we see that $N_v(t)$ must show significant oscillation in neighborhoods of $t_n'$.
Transferring this information to $u$ and using the upper bound on $N_u(t)$, we may conclude that $\liminf_{t\to-\infty} N_u(t) =
\liminf_{t\to\infty} N_u(t) =0$.  Thus we obtain a double high-to-low frequency cascade in the sense of Theorem~\ref{T:enemies}.

\medskip

Finally, we treat the case when $a(t)$ is strictly positive; we will construct a self-similar solution.

\medskip

{\bf Case 3:} $\lim_{T\to \infty} \osc(T) = \infty$ and $\inf_{t_0 \in J} a(t_0) = 2\eps >0$.  Let us call a $t_0\in J$
\emph{future-focusing} if $N_v(t)\geq\eps N_v(t_0)$ for all $t\geq t_0$; we call $t_0$ \emph{past-focusing} if $N_v(t)\geq\eps N_v(t_0)$
for all $t\leq t_0$.  Note that by hypothesis, every $t_0\in J$ is future-focusing, past-focusing, or possibly both.

Next we argue that either all sufficiently late times are future-focusing or all sufficiently early times are past-focusing.  If
this were not the case, one would be able to find arbitrarily long time intervals beginning with a future-focusing time and
ending with a past-focusing time.  The existence of such intervals would contradict the divergence of $\osc(T)$.
We restrict our attention to the case where all $t\geq t_0$ are future-focusing; the case when all sufficiently early times are past-focusing
can be treated symmetrically.

Choose $T$ so that $\osc(T) > 2\eps^{-1}$. We will now recursively construct an increasing sequence of times
$\{t_n\}_{n=0}^\infty$ so that
\begin{align}\label{t_n props}
0 < t_{n+1} - t_n \leq 2\eps^{-3} T N_v(t_n)^{-3} \quad\text{ and }\quad N_v(t_{n+1}) \geq 2 N_v(t_n).
\end{align}
Given $t_n$, set $t_n':=t_n + \eps^{-3}T N_v(t_n)^{-3}$.  Then
$$
J_n:=[t_n' - T N_v(t_n')^{-3},t_n' + T N_v(t_n')^{-3}] \subseteq [t_n,t_n+2\eps^{-3}T N_v(t_n)^{-3}].
$$
As $t_n$ is future-focusing, this allows us to conclude that $N_v(t)\geq \eps N_v(t_n)$ on $J_n$, but then by the way $T$ is chosen,
we may find $t_{n+1}\in J_n$ so that $N_v(t_{n+1})\geq 2 N_v(t_n)$.

Having obtained a sequence of times obeying \eqref{t_n props}, we may conclude that $t_n$ converge to a limit and $N_v(t_n)$ to
infinity.  Hence $\sup J$ is finite and $\lim_{n \to \infty} t_n =\sup J$.  Moreover, elementary manipulations using
\eqref{t_n props} and the local constancy property also yield
$$
\sup J -t \sim_v N_v(t)^{-3} \quad \text{for all} \quad t_0 \leq t < \sup J.
$$
Enlarging the compactness modulus function by a bounded amount, we may redefine
$$
N_v(t) = (\sup J-t)^{-1/3} \quad \text{for all} \quad t_0 \leq t < \sup J.
$$

Now consider the normalizations $v^{[t_n]}$.  After passing to a subsequence if necessary, $v^{[t_n]}$ converge locally
uniformly to a maximal-lifespan solution $u$ of mass $M(v)$ defined on an open interval $I$ containing $(-\infty,1)$, which is
almost periodic modulo symmetries.  Moreover, the frequency scale function of $u$ obeys
$$
N_u(s) \sim_v (1-s)^{-1/3} \quad \text{for all} \quad s\in(-\infty,1).
$$
Rescaling $u$ and applying a time translation (by $-1$) followed by a space/time reversal, we obtain our sought-after self-similar
solution.

This completes the proof of Theorem~\ref{T:enemies}. \qed

%
%
%
%

\end{document}